\documentclass{amsart}
\usepackage{graphicx}
\usepackage{amsmath}
\usepackage{amsthm}
\usepackage{amssymb}
\usepackage{tikz}
\usepackage[textwidth=3.5cm]{todonotes}

\newtheorem{thm}{Theorem}
\newtheorem{prop}[thm]{Proposition}
\newtheorem{lem}[thm]{Lemma}
\newtheorem{cor}[thm]{Corollary}

\newtheorem*{thm*}{Theorem}
\newtheorem*{lem*}{Lemma}

\theoremstyle{definition}
\newtheorem{defn}[thm]{Definition}
\newtheorem{rem}[thm]{Remark}
\newtheorem{ex}[thm]{Example}

\def\PA{\mathsf{PA}}

\def\ACA{\mathsf{ACA}}
\def\CT{\mathsf{CT}}
\def\ctm{\CT^-}
\def\csm{\mathsf{CS}^-}
\def\lpa{{\mathcal L}_{\PA}}

\newcommand{\anglebracket}[1]{\langle #1 \rangle}

\newcommand{\mc}[1]{\mathcal{#1}}

\newcommand{\restr}[1]{{\upharpoonright_{#1}}}

\DeclareMathOperator{\Def}{Def}

\DeclareMathOperator{\SSy}{SSy}

\DeclareMathOperator{\Cl}{Cl}

\DeclareMathOperator{\e}{\subseteq_\text{end}}

\DeclareMathOperator{\SInd}{SInd}
\DeclareMathOperator{\DC}{DC}
\DeclareMathOperator{\DNC}{DNC}
\DeclareMathOperator{\QC}{QC}

\DeclareMathOperator{\IDC}{IDC}
\DeclareMathOperator{\Sent}{Sent}
\DeclareMathOperator{\Form}{Form}
\DeclareMathOperator{\Prov}{Prov}
\DeclareMathOperator{\Comp}{Comp}

\DeclareMathOperator{\rk}{rk}
\DeclareMathOperator{\len}{len}
\DeclareMathOperator{\dc}{DC}

\begin{document}

\title{Pathologies in satisfaction classes}
\author{Athar Abdul-Quader}
\address{School of Natural and Social Sciences, 
Purchase College, SUNY}
\email{athar.abdulquader@purchase.edu}
\author{Mateusz {\L}e{\l}yk}
\address{Faculty of Philosophy, University of Warsaw}
\email{mlelyk@uw.edu.pl}

\begin{abstract}

We study subsets of countable recursively saturated models of $\PA$ which can be defined using pathologies in satisfaction classes. More precisely, we characterize those subsets $X$ such that there is a satisfaction class $S$ where $S$ behaves correctly on an idempotent disjunction of length $c$ if and only if $c \in X$. We generalize this result to characterize several types of pathologies including double negations, blocks of extraneous quantifiers, and binary disjunctions and conjunctions. We find a surprising relationship between the cuts which can be defined in this way and arithmetic saturation: namely, a countable nonstandard model is arithmetically saturated if and only if every cut can be the ``idempotent disjunctively correct cut" in some satisfaction class. We describe the relationship between types of pathologies and the closure properties of the cuts defined by these pathologies. 
\end{abstract}

\maketitle

{\Small \noindent \textbf{Keywords}: Nonstandard models of Peano Arithmetic, Satisfaction classes, Recursive saturation, Arithmetical saturation, Disjunctive correctness.}

{\Small \noindent \textbf{2020 \textit{Mathematics Subject Classification}}: 03C62, 03H15}

\section{Introduction}
Kotlarski-Krajewski-Lachlan famously showed \cite{kkl} that every countable, recursively saturated model of $\PA$ has a full satisfaction class. Enayat-Visser \cite{enayat-visser} strengthened this result using more typically model-theoretic tools. These results show the conservativity of the theory $\ctm$ of compositional truth over the base arithmetic theory $\PA$. Both proofs illustrate the weakness of $\ctm$: not only the theory is a conservative extension of the base theory $\PA$, but also it is consistent with failure of some very basic truth principles such as \emph{disjunctive correctness} ($\DC$): ``A disjunction is true if and only if it has a true disjunct". In particular one can construct models of $\ctm$ in which for a nonstandard number $a$ the disjunction
\[\underbrace{0\neq 0 \vee 0\neq 0 \vee \ldots \vee 0\neq 0}_{a \textnormal{ times }}\]
is within the scope of the truth predicate. Thus it is well known how to construct \emph{pathological} satisfaction classes.

One can easily exclude such pathological behaviour by adding to the theory $\ctm$ induction axioms for the extended language. It is well known that the theory $\CT$ of an \emph{inductive} truth predicate is not conservative over $\PA$; indeed, $\CT$ proves the Global Reflection Principle for $\PA$, that is the statement
\begin{equation}\label{grp}\tag{\textnormal{GRP}}
\forall \phi \bigl(\Prov_{\PA}(\phi)\rightarrow T(\phi)\bigr).
\end{equation}
In fact, $\CT_0$, the theory $\ctm$ augmented by $\Delta_0$-induction for formulas in the language including the truth predicate, is equivalent to \eqref{grp}. 

Recent work by Enayat and Pakhomov \cite{enayat-pakhomov} pointed to a deeper connection between non-conservativity and disjunctive correctness. 
The natural-looking extension of $\ctm$ with $\DC$ turns out to be equivalent to $\CT_0$. Ali Enayat (unpublished) separated $\DC$ into two principles: $\DC$-out, stating that every true disjunction has a true disjunct, and $\DC$-in, stating that a disjunction with a true disjunct is true. Cie{\'s}li{\'n}ski, {\L}e{\l}yk, and Wcis{\l}o \cite{cieslinski-lelyk-wcislo-dc} show that already $\ctm + \DC$-out is equivalent to $\CT_0$, while $\ctm + \DC$-in is conservative over $\PA$. Conservativity of $\DC$-in is shown by proving that every countable model of $\PA$ has an elementary extension which is ``disjunctively trivial": that is, one in which every disjunction of nonstandard length is evaluated as true. In such disjunctively trivial models of $\ctm$, $\omega$ is definable as the cut for which the truth predicate $T$ is ``disjunctively correct." 

In this article, we aim at deepening our understanding of the phenomenon of disjunctive correctness: we consider related questions around which sets can be definable by exploiting pathologies in the satisfaction class. We analyze ``local pathologies", along the lines of repeated (idempotent) disjunctions of a single, fixed sentence $\theta$, and non-local pathologies, where, for example, we consider idempotent disjunctions of all sentences. We completely classify the subsets of a model which are definable using local pathologies, and use this to conclude that a countable model of $\PA$ is arithmetically saturated if and only if it carries a satisfaction class which makes all disjunctions of nonstandard length true. We also classify the cuts in a model which can be definable using non-local pathologies.

From the definability perspective, our work complements that of \cite{wcislo_satisfaction_definability}, where it was shown that for every subset $A$ of a countable recursively saturated model $\mc{M}$ there is a satisfaction class $S$ such that $A$ is definable in $(\mc{M}, S)$ as (roughly speaking) the set of those numbers $x$ such that quantifier correctness fails on the $x$-th formula (in a  suitably chosen enumeration). We go in the reverse direction: starting from an idempotent sentential operation $F$ we ask when a set $A$ can be characterized as the set of those numbers $x$ for which the satisfaction class behaves correctly when $F$ is iterated $x$-times. Unlike in the case of \cite{wcislo_satisfaction_definability} it turns out that in some countable recursively saturated models, not every cut can be defined in this way. 

We conclude the paper with several properties about the full disjunctively correct cut.

\subsection{Preliminaries}

We formulate $\PA$ in the usual language $\lpa = \{ +, \times, <, 0, 1 \}$. We use script letters $\mc{M}, \mc{N}$, etc to denote models of $\PA$ and Roman letters $M, N,$ etc to denote their universes. $\textnormal{ElDiag}(\mc{M})$ denotes the elementary diagram of the model $\mc{M}$. We follow standard definitions and conventions used in the study of models of $\PA$: see \cite[Chapter 1]{ks}. We recall some of these conventions here. 

We fix standard coding for finite sets and sequences: for a model $\mc{M} \models \PA$, $a, b \in M$,
\begin{itemize}
\item $\len(a)$ denotes the length of the sequence coded by $a$,
\item $(a)_b$ denotes the $b$-th element of the sequence coded by $a$, and
\item we write $a \in b$ if $a$ is in the set coded by $b$.
\end{itemize}

\begin{defn}
A model $\mc{M}\models \PA$ is arithmetically saturated iff for every $a\in M$ for every type $p(x,a)$ which is arithmetically definable in the type of $a$, $p(x,a)$ is realized in $\mc{M}$.
\end{defn}

We note for the reader the equivalence between \emph{countable recursively saturated models} and \emph{countable resplendent models}, as well as the equivalence between \emph{arithmetically saturated models} and recursively saturated models in which $\omega$ is a strong cut. The interested reader is again directed to \cite{ks} for definitions and other references.

Let $\mc{M} \models \PA$. By $\Form^{\mc{M}}$ and $\Sent^{\mc{M}}$ we refer to the (definable) sets of (G\"{o}del codes of) formulas and sentences, respectively, in the sense of $\mc{M}$. For the rest of this article, we will not distinguish between a formula $\phi$ and its G\"{o}del code $\lceil \phi \rceil$. We use the following standard abbreviations:
 
\begin{itemize}
\item $\textnormal{Asn}(x,y)$ is an $\mc{L}_{\PA}$ formula which asserts that $y$ is an assignment for $x$, which means that it assigns values to all and only those variables which have free occurrences in $x$ ($x$ can be a term or a formula).
\item $s^{\alpha}$ denotes the value of the term $s$ under the assignment $\alpha$.
\item $\dot{\exists}$ denotes the arithmetical operation which given a variable $v$ and a formula $\phi$ returns $\exists v\phi$. $\dot{\vee}$, $\dot{\neg}$ and $\dot{=}$ have analogous  meanings.
\item for any two assignments $\alpha$, $\beta$, we write $\beta\sim_v\alpha$ iff $\beta$ differs from $\alpha$ at most on a variable $v$ and the domain of $\beta$ extends the domain of $\alpha$ at most with $v$.
\item for $\phi\in \Form_{\mc{L}_{\PA}}$, $\beta\restr{\phi}$ denotes the restriction of $\beta$ to the variables which have free occurrences in $\phi$.
\end{itemize}

Moreover we use a standard measure of complexity of formulae: we say that $\phi$ has complexity $x$, written $\textnormal{compl}(\phi)=x$, if $x$ is the maximal number of nested quantifiers and logical connectives. So the complexity of atomic formulae is $0$ and
\begin{align*}
    \textnormal{compl}(\neg \phi) =& \textnormal{compl}(\exists v\phi) = \textnormal{compl}(\phi)+1\\
    \textnormal{compl}(\phi\vee\psi) &= \max\{\textnormal{compl}(\phi), \textnormal{compl}(\psi)\}+1.
\end{align*}

\begin{defn}[$X$-restricted satisfaction class]
    The theory $\csm\restr{X}$ is the theory in $\lpa \cup \{ S,X \}$, where $S$ is a new binary predicate and $X$ a new unary predicate, expressing that $S$ is a compositional satisfaction class for formulae in the set $X$. Its axioms consist of the axioms for $\PA$ along with:

\begin{itemize}
\item[CS0] $\forall x\forall y \bigl(S(x,y)\rightarrow \Form_{\mc{L}_{\textnormal{PA}}}(x)\wedge \textnormal{Asn}(x,y)\bigr)$
\item[CS1] $\forall s,t\in \textnormal{Term}\forall \alpha \bigl(\textnormal{Asn}(s\dot{=}t, \alpha) \wedge X(s\dot{=}t)\rightarrow S((s\dot{=}t),\alpha)\equiv s^{\alpha} = t^{\alpha}\bigr).$
\item[CS2] $\forall \phi,\psi\in\Form_{\mc{L}_{\PA}} \forall \alpha\bigl(\textnormal{Asn}(\phi\dot{\vee}\psi,\alpha) \wedge X(\phi)\wedge X(\psi)\wedge X(\phi\dot{\vee}\psi)\rightarrow S(\phi\dot{\vee}\psi,\alpha)\equiv \bigl(S(\phi,\alpha\restr{\phi})\vee S(\psi,\alpha\restr{\psi})\bigr)\bigr).$
\item[CS3] $\forall \phi\in \Form_{\mc{L}_{\PA}}\forall \alpha\bigl(\textnormal{Asn}(\phi,\alpha)\wedge X(\phi)\wedge X(\dot{\neg}\phi)\rightarrow S(\dot{\neg}\phi,\alpha)\equiv \neg S(\phi,\alpha)\bigr)$
\item[CS4] $\forall \phi\in \Form_{\mc{L}_{\PA}}\forall v\in\textnormal{Var}\forall \alpha\bigl(\textnormal{Asn}(\dot{\exists}v\phi,\alpha) \wedge X(\phi)\wedge X(\dot{\exists}v\phi)\rightarrow S(\dot{\exists} v\phi,\alpha)\equiv \exists \beta\sim_v\alpha S(\phi,\beta\restr{\phi})\bigr).$
\end{itemize}
\end{defn}

By convention, when $(\mc{M}, S,X) \models \csm\upharpoonright_X$, $T(\phi)$ abbreviates $S(\phi, \emptyset)$. If $S\subseteq \mc{M}$ is such that $(\mc{M},S, X)\models \csm\upharpoonright_X$, then $S$ is called an \textit{$X$-satisfaction class} (on $\mathcal{M}$). If $S$ is an $X$-satsfiaction class on $\mc{M}$ and $A\subseteq M$, then by $S\restr{A}$ we  denote the set $S\cap(A\times M)$. In particular $S\restr{A}$ is (under the above assumptions) an $A\cap X$-satsifaction class on $\mc{M}$. $\csm$ extends $\csm\restr{X}$ with an axiom saying that $X$ is the whole universe. When talking about $\csm$ we omit the reference to $X$.  By definition we say that $S$ is a \textit{satisfaction class} (on $\mc{M}$) if $(\mc{M},S)\models \csm$. 

The above axioms of of $\csm$ are weak: in particular, it is known that they are not sufficient to guarantee the correct interaction between quantifiers and term substitutions. For example, it is an exercise to use an Enayat-Visser construction to show that in every countable and recursively saturated model $\mc{M}$ there is a satisfaction class $S$ such that for some formula $\phi$ and assignment $\alpha$, $(\exists v \phi,\alpha) \in S$ but for no closed term $t$, $\langle\phi[t/v],\alpha\rangle\in S$ ($\phi[t/v]$ denotes the substitution of a closed term $t$ for all occurrences of the variable $v$.)

Because of these and similar problems, it is not known whether in an arbitrary model of $(\mc{M}, S)\models \csm$ one can define a compositional truth predicate $T$ for the language of arithmetic satisfying the natural axiom
 \[\forall \phi(v)\bigl(T(\forall v\phi(v))\equiv \forall x T(\phi[\underline{x}/v])\bigr),\]
 where $\underline{x}$ denotes the canonical numeral naming $x$. It is known that each standard definition of truth from satisfation (e.g. ``being satisfied by all assignments" or ``being satisfied by an empty assignment") might fail to define a truth predicate in a model of $\csm$. 

To overcome these problems it is customary to extend the above list of axioms of $\csm$ with the regularity axiom (compare \cite{wcislo_satisfaction_definability}). Its full-blown definition is rather involved and we will give it in the Appendix. A satisfaction class which satisfies the regularity axiom is called a \textit{regular} satisfaction class. Importantly, if $S$ is a regular satisfaction class in $\mc{M}$, then terms with the same values can be substituted for free variables in a formula salva veritate, i.e. for every formula $\phi\in\Form^{\mc{M}}$, every variable $v\in \textnormal{Var}^{\mc{M}}$, all terms $s,t\in \textnormal{Term}^{\mc{M}}$ and all assignments $\alpha$ it holds in $(\mc{M}, S)$ that 
     \[\textnormal{Asn}(\phi([t/v]),\alpha)\wedge \textnormal{Asn}(\phi[s/v],\alpha)\wedge s^{\alpha} = t^{\alpha} \rightarrow S(\phi[s,v],\alpha)\equiv S(\phi[t/v],\beta).\]
    One can check, that if $S$ is a regular satisfaction class in $\mc{M}$, then the formula $\textnormal{Sent}(x)\wedge S(x,\emptyset)$ defines in $(\mc{M},S)$ a truth predicate which satisfy the above natural axiom for the universal quantifier.
    
 In the Appendix we show how to improve one of our constructions in order to obtain regular satisfaction classes. As a consequence we will be able to construct many pathological \textit{truth} classes. However, we decided to leave the regularization of all our constructions for further research.

Another basic property of satisfaction classes is satisfying internal induction. Before introducing it let us define one handy abbreviation: if $(\mc{M},S)\models \csm$, and $\psi$ is a formula in the sense of $\mc{M}$ with exactly one free variable, then $T*\psi(x)$ denotes a $\mc{L}_{\PA}\cup\{S\}$-formula with one free variable $x$ which naturaly expresses ``The result of substituting the numeral naming $x$ for the unique free variable of $\psi$ is satisfied by the empty assignment"(see \cite{wcislyk_mpt}, Lemma 3.6.) We say that in $(\mc{M}, S)\models\csm,$ $S$ satisfies the internal induction iff for every $\psi\in \Form^{\mc{M}}$ with a unique free variable, the formula $T*\psi(x)$ satisfies the induction axiom, i.e.
 \[(\mc{M},S)\models T*\psi(0)\wedge \forall x \bigl(T*\psi(x)\rightarrow T*\psi(x+1)\bigr) \rightarrow \forall x T*\psi(x).\]
 We conjecture that all our constructions can be fine-tuned to yield regular satisfaction classes satisfying internal induction, however we decided to treat this problem on a different occasion.

\begin{rem}\label{rem_int_ind}
As shown in \cite{wcislyk_mpt}, Lemma 3.7, if $(\mc{M}, S) \models \csm$, $S$ is regular and satisfies internal induction, then for every $\psi\in\Form^{\mc{M}}$ with exactly one free variable, if $X_{\psi} = \{x\in M \ \ : \ \ (\mc{M},S)\models T*\psi(x)\}$, then $(\mc{M},X_{\psi})\models \PA^*$. That is, $(\mc{M}, X_\psi)$ satisfies the full induction schema in the language $\lpa \cup \{ X \}$, where $X$ is interpreted as $X_\psi$.
 \end{rem}

 \begin{defn}[Local compositional conditions]
     Let $\textnormal{Comp}(x,y,z)$ be the disjunction of the following $\mc{L}_{\PA}\cup\{S\}$ formulae
     \begin{enumerate}
         \item $\exists s,t\in\textnormal{Term} (x = (s\dot{=}t)\wedge \forall \alpha (\textnormal{Asn}(x,\alpha)\rightarrow S(x,\alpha)\equiv s^{\alpha} = t^{\alpha}))$.
         \item $x = (y\dot{\vee} z) \wedge \forall \alpha(\textnormal{Asn}(x,\alpha) \rightarrow S(x,\alpha)\equiv (S(y,\alpha\restr{y})\vee S(z,\alpha\restr{z})))$.
         \item $x= (\dot{\neg}y)\wedge \forall \alpha(\textnormal{Asn}(x,\alpha)\rightarrow S(x,\alpha)\equiv\neg S(y,\alpha))$.
         \item $\exists v\in \textnormal{Var} (x = \dot{\exists}vy\wedge \forall \alpha(\textnormal{Asn}(x,\alpha)\rightarrow S(x,\alpha)\equiv \exists\beta\sim_v\alpha S(y,\beta\restr{y})))$.
     \end{enumerate}
\end{defn}

Suppose $\anglebracket{\phi_i : i \leq c}$ is a coded sequence of elements of $\Sent^{\mc{M}}$ and suppose $\theta \in \Sent^{\mc{M}}$. 

\begin{itemize}
\item $\bigvee\limits_{i \leq c} \phi_i$ is defined, inductively, so that $\bigvee\limits_{i \leq 0} \phi_i = \phi_0$, and $\bigvee\limits_{i \leq n + 1} \phi_i = (\bigvee\limits_{i \leq n} \phi_i) \vee \phi_{n+1}$.
\item $\bigwedge\limits_{i \leq c} \phi_i$ is defined similarly.

\end{itemize}

Given an $\mc{M}$-definable function $F:\Sent^{\mc{M}}\rightarrow \Sent^{\mc{M}}$, we define $F^c(x)$ by induction on $c$ as follows: $F^0(x) = x$, $F^{c+1}(x) = F(F^c(x))$.
\begin{itemize}
\item $\bigvee\limits^{\text{bin}}_{i \leq c} \theta$ is defined as $F_{\vee}^c(\theta)$, where $F_{\vee}(\phi) = \phi\vee\phi$. These are ``binary idempotent disjunctions." Similarly, one can define ``binary idempotent conjunctions."
\item $(\lnot \lnot)^{c } \theta$ is defined as $F_{\lnot\lnot}^c(\theta)$, where $F_{\lnot\lnot}(\phi) = \neg\neg\phi$.
\item $(\forall x)^{c } \theta$ is defined as $F_{\forall}^c(\theta)$, where $F_{\forall}(\phi) = \forall x \phi$.
\end{itemize}

In a model $(\mc{M}, S) \models \csm$, we can define the following sets:

\begin{itemize}
\item for a given $\theta$, the ``idempotent disjunctively correct set for $\theta$", $$\IDC^{\theta}_S = \{ c : T(\bigvee\limits_{i < c} \theta) \equiv T(\theta) \},$$
\item the ``idempotent disjunctively correct set": $$\IDC_S = \{ c : \forall \phi T(\bigvee\limits_{i < c} \phi) \equiv T(\phi) \}.$$
\item the ``disjunctively correct set": $$\dc_S = \{c\in M : (\mc{M}, S)\models \forall \anglebracket{\phi_i: i\leq c} \bigl(T(\bigvee_{i\leq c} \phi_i)\equiv \exists i\leq c T(\phi_i)\bigr)\}.$$
\end{itemize}

We can similarly define the ``conjunctively correct set" for a given $\theta$, the ``quantifier correct set" for a given $\theta$ ($\QC^{\theta}_S$), the ``double negations correct set" for a given $\theta$ ($\DNC^{\theta}_S$), the ``binary idempotent disjunctively/conjunctively correct set" ($\IDC^{\text{bin}, \theta}_S$), or their respective non-local versions ($\QC_S, \DNC_S, \IDC^{\text{bin}}_S$).

Given a set $X$ (often one of the above pathologically definable sets), we introduce the following notation for \emph{the longest initial segment of $X$}: $$I(X) = \{ x \in X : \forall y \leq x (y \in X) \}.$$ This allows us to denote, for example, the idempotent disjunctively correct \emph{cut}, $I(\IDC_S)$.



\section{Separability}\label{sep-section}
In this part, we classify which sets can be $\IDC^{0 = 1}_S$ for some $S$. Rather than simply looking at disjunctions, however, we generalize the setting to draw similar conclusions about the conjunctively correct set for $0 = 0$, the double negations correct set for any atomic sentence $\phi$, or the binary idempotent disjunctively / conjunctively correct set for $\phi$ and much more.

\begin{defn}\label{sep}
Let $X \subseteq \Form^{\mc{M}}$.
\begin{enumerate}
	\item If $x, y \in \Form^{\mc{M}}$, we say $x \triangleleft y$ if $x$ is an immediate subformula of $y$.
	\item $X$ is \emph{closed} if whenever $x \triangleleft y \in X$, then $x \in X$.
	\item $\Cl(X)$ is the smallest closed set containing $X$.
	\item $F \subseteq X$ \emph{generates} $X$ if $X = \Cl(F)$.
	\item $X$ is \emph{finitely generated} if there is a finite $F \subseteq X$ that generates it.
\end{enumerate}
\end{defn}	
We describe a generalization of the idempotent disjunction operation $c \mapsto \bigvee\limits_{i < c} \theta$.
\begin{defn}\label{local-id}
    Fix a standard sentence $\theta$. Let $\Phi(p, q)$ be a (finite) propositional template, over propositional variables $p$ and $q$. By this we mean that in $\Phi$ we allow all propositional connectives, along with quantifiers (over dummy variables). We insist that $\Phi$ has non-zero complexity (that is, $\Phi(p, q)$ has at least one propositional connective or quantifier), along with the following properties:
\begin{itemize}
    \item $q$ appears in $\Phi(p, q)$,
    \item if $\mc{M} \models \theta$, then $\Phi(\top, q)$ is equivalent to $q$, and
    \item if $\mc{M} \models \lnot \theta$, then $\Phi(\bot, q)$ is equivalent to $q$.
\end{itemize} Define $F : M \to \Sent^{\mc{M}}$ as follows:
\begin{itemize}
    \item $F(0) = \theta$, and
    \item $F(x + 1) = \Phi(\theta, F(x))$.
\end{itemize} We say such an $F$ is a \emph{local idempotent sentential operator for $\theta$}, and $\Phi(p, q)$ is a \emph{template} for $F$.
\end{defn}
We emphasize here that $\Phi$ is finite, so that if $\phi$ and $\psi$ are sentences, then $\psi \in \Cl(\Phi(\phi, \psi))$. In addition, if $p$ appears in $\Phi(p, q)$, then $\phi \in \Cl(\Phi(\phi, \psi))$ as well. 

Note that for any $n \in \omega$ and atomic sentence $\theta$, if $F$ is a local idempotent sentential operator for $\theta$ and $(\mc{M}, S) \models \csm$, then $(\mc{M}, S) \models T(\theta) \equiv T(F(n))$. In fact, $(\mc{M}, S) \models T(F(x)) \equiv T(F(x + n))$, for each $x \in M$.

This approach allows us to generalize several examples of local pathologies, for example:
\begin{align*}
\left\{ \bigvee\limits_c (0 \neq 0) : c \in M \right\}, & &
 \left\{ \bigwedge\limits_c (0 = 0) : c \in M \right\},\\
\left\{ (\forall x)^c (0 = 0) : c \in M \right\}, & &
\left\{ (\lnot \lnot)^c (0 = 0) : c \in M \right\}
\end{align*} 
can all appear as $\{ F(c) : c \in M \}$ for various $\theta$ and $\Phi$. We study the question of when, given such a function $F$, a set $X$ can be the set $\{ x : T(F(x)) \equiv T(\theta) \}$ in a model $(\mc{M}, S) \models \csm$. We will see that such sets $X$ will require the following property.

\begin{defn}
Let $\mc{M} \models \PA$, and $A \subseteq D \subseteq M$. $A$ is \emph{separable from $D$} if for each $a$ such that for every $n \in \omega$, $(a)_n \in D$, there is $c$ such that for each $n \in \omega$, $(a)_n \in A$ if and only if $\mc{M} \models n \in c$. We say a set $X$ is \emph{separable} if it is separable from $M$.
\end{defn}

In Propositions \ref{def-are-sep}, \ref{prop_closure_sep_preim}, and \ref{prop_supremum_separabilit}, we refer to definable sets and functions. Here we insist that these are definable in the arithmetic structure of $\mc{M}$: that is, they are definable (possibly using parameters) using formulas from $\lpa$. First we notice some basic properties of separability.

\begin{prop}\label{def-are-sep}
Let $\mc{M} \models \PA$. Suppose $D_1, D_2$ are $\mc{M}$-definable, $A\subseteq D_1\cap D_2$, and $A \neq D_1, D_2$. Then $A$ is separable from $D_1$ iff $A$ is separable from $D_2$.
\end{prop}
\begin{proof}
Fix $d\in D_1\setminus A$. Assume $A$ is separable from $D_1$ and fix any $a$ such that for every $n$, $(a)_n\in D_2$. Let $b$ be defined by
\begin{displaymath}
(b)_i = \left\{\begin{array}{l} (a)_i \textnormal{ if } (a)_i\in D_1\\
                              d \textnormal{ otherwise } \end{array}\right.  
\end{displaymath}
Then for every $i\in\omega$, $(b)_i\in D_1$, so there is $c$ such that for every $i\in\omega$, $(b)_i \in A$ iff $i\in c$. Then it follows that also for every $i\in\omega$, $(a)_i \in A$ iff $i\in c$.
\end{proof}

\begin{prop}\label{prop_closure_sep_preim}
Let $\mc{M} \models \PA$. Suppose $A\subseteq D$ and $f$ is an $\mc{M}$-definable function such that $D\subseteq\text{im}(f)$. Then if $A$ is separable from $D$, then $f^{-1}[A]$ is separable from $f^{-1}[D]$
\end{prop}
\begin{proof}
Easy exercise.
\end{proof}

\begin{prop}\label{prop_supremum_separabilit}
Let $\mc{M} \models \PA$. Let $I\subseteq_e \mc{M}$ and $A\subseteq \mc{M}$ be an $\mc{M}$-definable set such that $\sup(A\cap I) = I$ and $A\cap I$ is separable. Then $I$ is separable. 
\end{prop}
\begin{proof}
Define the function $f$ by 

\begin{displaymath}
f(x) = \left\{\begin{array}{l} \mu y. \{y\in A : x\leq y\} \textnormal{ if such $y$ exists }\\
                              0 \textnormal{ otherwise } \end{array}\right.  
\end{displaymath}
Then, by the assumptions, $I = f^{-1}[A\cap I]$. The result follows by Proposition \ref{prop_closure_sep_preim}.
\end{proof}

As stated before, given $\theta$, a local idempotent sentential operator $F$ for $\theta$, and $D = \{ F(x) : x \in M \}$, we wish to classify the subsets $A \subseteq D$ which can be the sets of true sentences in $D$ (equivalently, we wish to classify the sets $X$ such that $\{ F(x) : x \in X \}$ is the set of true sentences in $D$). First we need the following Lemma.

\begin{lem}\label{base-for-cl-f-iterates}
Let $\mc{M} \models \PA$. Let $\theta$ be an atomic sentence and $F$ a local idempotent sentential operator for $\theta$. Let $J_0, J_1 \subseteq M$ be closed under predecessors, disjoint, and $J_i \cap \omega = \emptyset$ for $i = 0, 1$. Let $X = \Cl(\{ F(x) : x\in J_0\cup J_1 \})$. Then there is a unique $X$-satisfaction class $S$ such that for each $i$ and $x \in J_i$, $(F(x), \emptyset) \in S$ if and only if $i = 0$.\end{lem}

\begin{proof}
Let $S_0 = \{ (F(x), \emptyset) : x \in J_0 \}$. We extend $S_0$ to an $X$-satisfaction class $S$. Take any $\phi\in X$. Then, since $J_i$ are closed under predecessors and disjoint, then there is a unique $i$ and minimal $x$ such that $\phi \in \Cl(F(x)) $ and $x\in J_i$. Recall that $F(x) = \Phi(\theta, F(x - 1))$, and $\theta$ is atomic. One notices that the subformulas of $\Phi(\theta, q)$ must be equivalent to one of $q$, $\lnot q$, $\top$, or $\bot$. Let $\Psi(p, q)$ be the subformula of $\Phi(p, q)$ such that $\Psi(\theta, F(x - 1)) = \phi$. Again, the presentation of $\phi$ as $\Psi(\theta, F(x-1))$ is unique by induction in $\mc{M}$. We put $\anglebracket{\phi, \emptyset} \in S$ if any of the following hold:
\begin{itemize}
\item $\Psi(\theta, q)$ is equivalent to $q$ and $i = 0$, 
\item $\Psi(\theta, q)$ is equivalent to $\lnot q$ and $i = 1$, or
\item $\Psi(\theta, q)$ is equivalent to $\top$.
\end{itemize} One checks that $S$ is an $X$-satisfaction class.
\end{proof}

Theorems \ref{truth-closed-theorem} and \ref{separability-theorem} are generalizations of unpublished work by Jim Schmerl\footnote{Private communication to Ali Enayat.}.

\begin{thm}\label{truth-closed-theorem}Let $\mc{M} \models \PA$ be countable and recursively saturated. Let $\theta$ be an atomic sentence and $F$ a local idempotent sentential operator for $\theta$. Let $X \subseteq M$ be separable, closed under successors and predecessors, and for each $n \in \omega$, $n \in X$ if and only if $\mc{M} \models \theta$. Then $\mc{M}$ has an expansion $(\mc{M}, S) \models \csm$ such that $X = \{ x \in M : (\mc{M}, S) \models T(F(x)) \equiv T(\theta) \}$.
\end{thm}
Notice that $X$ is separable if and only if $M \setminus X$ is separable. This means that there is some flexibility in building such satisfaction classes $S$.

\begin{proof}
Let $D = \{ F(x) : x \in M \}$ and $A = \{ F(x) : x \in X \}$. Note that $A$ is separable from $D$. We build sequences $F_0 \subseteq F_1 \subseteq \ldots$ and $S_0, S_1, \ldots$ such that:
\begin{itemize}
	\item each $F_i$ is a finitely generated set of formulas such that $\cup F_i = \Form^{\mc{M}}$,
	\item each $S_i$ is a full satisfaction class and $(\mc{M}, S_i)$ is recursively saturated,
	\item $S_{i+1} \restriction F_i = S_i \restriction F_i$, and
	\item for each $\phi \in D \cap F_i$, $(\phi, \emptyset) \in S_i$ if and only if $\phi \in A$.
\end{itemize}
Given such a sequence, $S = \cup S_i \restriction F_i$ would be the required full satisfaction class on $\mc{M}$.

Externally fix an enumeration of $\Form^{\mc{M}}$ in order type $\omega$. We can assume, without loss of generality, that $\theta$ appears first in this enumeration. Suppose $F_i$ and $S_i$ have been constructed. Let $F_{i+1}$ be generated by $F_i$ and the least $x \in \Form^{\mc{M}} \setminus F_i$ in the aforementioned enumeration. Let $F^\prime = F_i \cup (F_{i+1} \cap D)$. Let $a$ be a sequence such that $\{ F((a)_n) : n \in \omega \} = F_{i+1} \cap D$. Note that this is possible since $F_{i+1}$ is finitely generated. Let $c$ be as in the definition of separability for $a$. 

Since $X$ is closed under successors and predecessors, then if $(a)_n$ and $(a)_m$ are in the same $\mathbb{Z}$-gap (that is, there is some $k \in \omega$ such that $(a)_n$ and $(a)_m$ differ by $k$), then $(a)_n \in X$ if and only if $(a)_m \in X$. Since $X$ is separable, this means that, if $(a)_n$ and $(a)_m$ are in the same $\mathbb{Z}$-gap, then $n \in c$ if and only if $m \in c$. Let $J_0$ be the closure under successors and predecessors of $\{ (a)_n : n \in \omega, n \in c,$ and $(a)_n > \omega \}$, and $J_1$ be the closure under successors and predecessors of $\{ (a)_n : n \in \omega, n \notin c,$ and $(a)_n > \omega \}$. By Lemma \ref{base-for-cl-f-iterates}, there is a $\Cl(F_{i+1} \cap D)$-satisfaction class $S^\prime$ such that for each $\phi = F((a)_n) \in F_{i+1} \cap D$, $S^\prime(F((a)_n, \emptyset)$ if and only if $(a)_n \in X$. That is, $S^\prime(\phi, \emptyset)$ if and only if $\phi \in A$.

Notice that $\Cl(F^\prime) = F_i \cup \Cl(F_{i+1} \cap D)$. We extend $S^\prime$ to a $\Cl(F^\prime)$ satisfaction class simply by preserving $S_i$ on $F_i$. One notices that if $\phi \in F_i \cap D$, then by induction $\anglebracket{\phi, \emptyset} \in S_i$ if and only if $\phi \in A$.

Then $S^\prime$ is a $\Cl(F^\prime)$-satisfaction class, so by \cite[Lemma 3.1]{enayat-visser}, $\mc{M}$ has an elementary extension $\mc{N}$ carrying a $\Form^{\mc{M}}$-satisfaction class $S$ agreeing with $S^\prime$ on $\Cl(F^\prime)$. In particular, this shows the consistency of the recursive theory $Th$ consisting of the following:
\begin{itemize}
	\item $S$ is a full satisfaction class,
	\item $\{ S(\phi, \alpha) \equiv S_i(\phi, \alpha) : \phi \in F_i \}$, and
	\item $\{ S(F((a)_n), \emptyset) \equiv n \in c : n \in \omega \}$.
\end{itemize} Since $(\mc{M}, S_i)$ is recursively saturated, by resplendency $(\mc{M}, S_i)$ has an expansion to $Th$, and such an expansion is a full satisfaction class agreeing with $S^\prime$ on formulas from $\Cl(F^\prime)$. Recall that countable recursively saturated models are \emph{chronically resplendent} (\cite[Theorem 1.9.3]{ks}): by this we mean that such expansions can, themselves, be taken to be resplendent. That is, we can assume that $(\mc{M}, S_i, S)$ is recursively saturated. Let $S_{i+1} = S$ and continue.
\end{proof}

In the above result, notice that if $n \in \omega$, then clearly $\mc{M} \models F(n)$ if and only if $\mc{M} \models \theta$. Therefore, $\omega \subseteq X$ if and only if $\mc{M} \models \theta$, and $\omega \cap X = \emptyset$ if and only if $\mc{M} \models \lnot \theta$. Moreover, if $X = \{ x : (\mc{M}, S) \models T(F(x)) \}$ then $X$ is necessarily closed under successors and predecessors. The next result shows that separability of $X$ is also necessary. 

\begin{thm}\label{separability-theorem}
Suppose $(\mc{M}, S) \models \csm$, $D$ is any set of sentences (not necessarily of the form $\{ F(x) : x \in M \}$), and $A = \{ \phi \in D : (\mc{M}, S) \models T(\phi) \}$. Then $A$ is separable from $D$.
\end{thm}

Note that by Proposition \ref{prop_closure_sep_preim}, if $D = \{ F(x) : x \in M \}$, $A = \{ F(x) : (\mc{M}, S) \models T(F(x)) \}$, and $X = \{ x : F(x) \in A \}$, this is equivalent to stating that $X = \{ x : F(x) \in A \}$ is separable (from $M$).

\begin{proof}
Let $a \in M$ be such that for each $n \in \omega$, $(a)_n \in D$. We show that that there is a $c$ so that for all $i \in \omega$, $(a)_i \in A$ iff $i \in c$.

By a result of Stuart Smith, \cite[Theorem 2.19]{smith-nonstandard-def}, $(\mc{M}, S)$ is \emph{definably $S$-saturated}. This means that for any coded sequence $\anglebracket{\phi_i(x) : i \in \omega}$ such that each $\phi_i \in \Form^{\mc{M}}$, if for each $i \in \omega$ there is $m \in M$ such that $(\mc{M}, S) \models \forall j < i (T(\phi_j(m)))$, then there is $m \in M$ such that for all $i \in \omega$, $(\mc{M}, S) \models T(\phi_i(m))$.

Let $\phi_j(x)$ be the formula given by $(a)_j  \equiv (j \in x)$. That is, since $(a)_j$ is the code of a sentence, $\phi_j(m)$ is evaluated as true in a satisfaction class $S$ if the sentence $(a)_j$ is evaluated as true and $j \in m$, or $(a)_j$ is evaluated as false and $j \not \in m$.

Let $i \in \omega$, and let $m \in M$ be such that for all $j < i$, $(a)_j \in A$ if and only if $j \in m$. Then, $$(\mc{M}, S) \models \forall j \leq i \: (T(\phi_j(m))).$$ Therefore there is $m$ such that for all $i \in \omega$, $(\mc{M}, S) \models T(\phi_i(m))$. In particular, for each $n \in \omega$, if $(a)_n \in A$, then $T( (a)_n)$ and therefore $n \in m$. Moreover, if $n \not \in m$, then $(\mc{M}, S) \models \lnot T((a)_n)$. By assumption this means $(a)_n \not \in A$.
\end{proof}

\section{Separable Cuts}\label{wsr-section}
In this section, we wish to examine the results of the previous section in case where we have $I \e M$ a cut. We examine some properties of separable cuts. We conclude this section by showing that a countable model is arithmetically saturated if and only if it has a disjunctively trivial expansion to a model of $\csm$.

\begin{prop}\label{sep-iff-wsr}Let $\mc{M} \models \PA$ be nonstandard and  $I \e M$. The following are equivalent.
\begin{enumerate}
    \item\label{sep-cut} $I$ is separable.
    \item\label{weakly-superrational} There is no $a \in M$ such that $I = \sup(\{ (a)_i : i \in \omega \} \cap I) = \inf(\{ (a)_i : i \in \omega \} \setminus I)$.
    \item\label{weakly-strong} For every $a \in M$, there is $d$ such that for all $i \in \omega$, $(a)_i \in I$ if and only if $(a)_i < d$.
\end{enumerate}
\end{prop}

Compare $\eqref{weakly-strong}$ to the notion of \emph{strength}: a cut $I \e M$ is strong if for each $a$ there is $c > I$ such that whenever $i \in I$, $(a)_n \in I$ if and only if $(a)_n < c$. Clearly, condition $\eqref{weakly-strong}$ is equivalent to strength if $I = \omega$. 

\begin{proof}
 $\eqref{weakly-superrational} \iff \eqref{weakly-strong}$ follows immediately from definitions. 

We show $\eqref{sep-cut} \implies \eqref{weakly-strong}$: Suppose $I$ is separable and let $a \in M$. We show that there is $c \in M$ such that for each $n \in \omega$, $(a)_n \in I$ if and only if $(a)_n < c$. Since $I$ is separable, there is $c$ such that for each $n \in \omega$, $(a)_n \in I$ if and only if $n \in c$. Consider the type $$p(x) = \{ (a)_n < x \equiv n \in c : n \in \omega \}.$$ This type is finitely satisfiable, so (by restricted saturation of nonstandard models, see \cite[Corollary 1.11.4]{ks}) there is $c^\prime$ which satisfies $p(x)$.

Now we show $\eqref{weakly-strong} \implies \eqref{sep-cut}$. Let $a \in M$. There is $c$ such that $(a)_n \in I$ if and only if $(a)_n < c$. Consider the type $$p(x) = \{ (a)_n < c \equiv n \in x : n \in \omega \}.$$ This type is finitely satisfiable and therefore satisfied by some $c^\prime \in M$. Such a $c^\prime$ witnesses separability of $I$.
\end{proof}

By Theorem \ref{truth-closed-theorem}, Theorem \ref{separability-theorem}, and Proposition \ref{sep-iff-wsr}, then  $I$ is separable if and only if there is $S$ such that $(\mc{M}, S) \models \csm$ and $$I = \{ c : (\mc{M}, S) \models \lnot T(\bigvee\limits_c (0 = 1)) \}.$$ Suppose $(\mc{M}, S) \models \csm$. By Theorem \ref{separability-theorem} one has that $X = \{ c : (\mc{M}, S) \models \lnot T(\bigvee\limits_c (0 = 1)) \}$ is separable. Is it the case that $I(\IDC^{0=1}_S) = \{ x : \forall c < x \: \lnot T(\bigvee\limits_c (0 = 1)) \}$ is also separable? Our next result shows that it is not always the case: if $I \e M$ has no least $\mathbb{Z}$-gap above it, then there is $S$ such that $(\mc{M}, S) \models \csm$ and $I(\IDC^{0=1}_S) = I$. Later, in Corollary \ref{i(x)-not-sep}, we see that if $\mc{M}$ is not arithmetically saturated, then such an $I$ need not be separable.

\begin{prop}\label{no-least-z-0=1}Let $\mc{M} \models \PA$ be countable and recursively saturated.
Suppose $I \e M$ has no least $\mathbb{Z}$-gap. Then there is $S$ such that $(\mc{M}, S) \models \csm$ and $$I = \{ x : \forall c < x\:  \lnot T(\bigvee\limits_c(0 = 1)) \}.$$
\end{prop}

\begin{proof}
First notice that for any $c < d$ in different $\mathbb{Z}$-gaps, for any $a \in M$, there is $b$ such that $c < b < d$ and $b \not \in \{ (a)_i : i \in \omega \}$. To see this, notice that if $a, c,$ and $d$ are as above, by recursive saturation the type $$p(x) = \{ c < x < d \} \cup \{ (a)_i \neq x : i \in \omega \}$$ is realized in $M$. In fact, one can ensure that the $\mathbb{Z}$-gap of such a $b$ is disjoint from $c$, $d$, and $\{ (a)_i : i \in \omega \}$.

Now we show how to construct the required satisfaction class. Fix a sequence $d_0 > d_1 > \ldots$ such that $d_{i+1}$ is not in the same $\mathbb{Z}$-gap as $d_i$ and $\inf(\{ d_i : i \in \omega \}) = I$. We proceed similarly to Theorem \ref{truth-closed-theorem}: we build sequences $b_0 > b_1 > \ldots$, $F_0 \subseteq F_1 \subseteq \ldots$ and $S_0, S_1, \ldots$ such that:
\begin{itemize}
    \item for each $i \in \omega$, $d_{i+1} < b_i < d_i$ and $b_i$ is in a different $\mathbb{Z}$-gap from $d_i$ and $d_{i+1}$,
    \item each $F_i$ is a finitely generated set of formulas such that $\cup F_i = \Form^{\mc{M}}$,
    \item each $S_i$ is a full satisfaction class and $(\mc{M}, S_i)$ is recursively saturated,
    \item $S_{i+1} \restriction F_i = S_{i} \restriction F_i$,
    \item $\bigvee\limits_{d_i} (0 = 1) \in F_i$ and whenever $\bigvee\limits_c (0 = 1) \in F_i$ and $c \leq d_i$, $\anglebracket{\bigvee\limits_{c} (0 = 1), \emptyset} \not \in S_{i}$, and 
    \item $\bigvee\limits_{b_i} (0 = 1) \in F_{i+1} \setminus F_i$ and $\anglebracket{\bigvee\limits_{b_i} (0 = 1), \emptyset} \in S_{i+1}$.
\end{itemize}

Given such a sequence, let $S = \cup (S_{i} \restriction F_i)$. Then $S$ is the required full satisfaction class. To see this, suppose $J = \{ x : \forall c < x \: \lnot T(\bigvee\limits_c (0 = 1)) \}$. Notice that $(\mc{M}, S) \models T(\bigvee\limits_{b_i} (0 = 1))$, so for each $x \in J$ and $i \in \omega$, $x < b_i$; since $\inf(\{ b_i : i \in \omega \}) = I$, we have $J \subseteq I$. Conversely, let $d \in I$. For each $c < d$, there is $i$ such that $\bigvee\limits_{c} (0 = 1) \in F_i$. Then $c < d_i$, so $\anglebracket{\bigvee\limits_c (0 = 1), \emptyset} \not \in S_{i}$, and hence $\anglebracket{\bigvee\limits_c (0 = 1), \emptyset} \not \in S$.

We proceed to the construction. Suppose $F_i$ and $S_{i}$ has been constructed satisfying the above. Since $F_i$ is finitely generated, there is $a$ coding the lengths of disjunctions of $(0 = 1)$ in $F_i$. By recursive saturation, there is $b_i$ such that $d_{i+1} < b_i < d_i$ and $b_i \not \in \{ (a)_i : i \in \omega \}$; moreover, we ensure that the $\mathbb{Z}$-gap of $b_i$ is disjoint from $d_i$, $d_{i+1}$, and $\{ (a)_i : i \in \omega \}$. Let $F_{i+1}$ be generated by $F_i$, $\bigvee\limits_{b_i} (0 = 1)$, $\bigvee\limits_{d_{i+1}}(0 = 1)$, and the first formula $\phi \not \in F_i$ in some externally fixed enumeration of $\Form^{\mc{M}}$. Let $$F^\prime = F_i \cup (F_{i+1} \cap \{ \bigvee\limits_c (0 = 1) : c \in M \}).$$ Then $F^\prime$ is a closed set of formulas. Let $S^\prime = S_{i} \restriction F_i \cup \{ \anglebracket{\bigvee\limits_{b_i - n} (0 = 1), \emptyset} : n \in \omega \}$. In particular, $\anglebracket{\bigvee\limits_{d_{i+1}}(0 = 1), \emptyset} \not \in S^\prime$. $S^\prime$ is an $F^\prime$-satisfaction class, so by \cite[Lemma 3.1]{enayat-visser}, $\mc{M}$ has an elementary extension $\mc{N}$ carrying a $\Form^{\mc{M}}$-satisfaction class $S$ agreeing with $S^\prime$ on $F^\prime$. Therefore, the theory $Th$ asserting the following is consistent:
\begin{itemize}
    \item $S$ is a full satisfaction class,
    \item $S$ agrees with $S_{i}$ on formulas from $F_i$, 
    \item $\{ S(\bigvee\limits_{b_i - n}(0 = 1), \emptyset) : n \in \omega \}$, and
    \item $\{ \lnot S(\bigvee\limits_{c}(0 = 1), \emptyset) : c < d_{i+1}, \bigvee\limits_{c}(0 = 1) \in F_{i+1} \}$.
\end{itemize} By resplendency, $\mc{M}$ has a full satisfaction class $S$ satisfying $Th$; by chronic resplendency, we can assume $(\mc{M}, S)$ is recursively saturated. Let $S_{i+1} = S$ and continue. 
\end{proof}

To find some examples of separable cuts, we recall some definitions from \cite{kossak-omega}. Below, we let $\Def_0(a)$ be the set of elements of $\mc{M}$ which are $\Delta_0$-definable from $a$ in $\mc{M}$.

\begin{defn}[{\cite{kossak-omega}}] Let $\mc{M} \models \PA$ and let $I \e M$. 
	\begin{enumerate}
		\item $I$ is \emph{coded by $\omega$ from below} if there is $a \in M$ such that $I = \sup( \{ (a)_i : i \in \omega \})$. $I$ is \emph{coded by $\omega$ from above} if there is $a \in M$ such that $I = \inf(\{ (a)_i : i \in \omega \})$. $I$ is \textit{$\omega$-coded} if it is either coded by $\omega$ from below or from above.
		\item $I$ is \emph{$0$-superrational} if there is $a \in M$ such that either $\Def_0(a) \cap I$ is cofinal in $I$ and for all $b \in M$, $\Def_0(b) \setminus I$ is not coinitial in $M \setminus I$, or $\Def_0(a) \setminus I$ is coinitial in $M \setminus I$ and for all $b \in M$, $\Def_0(b) \cap I$ is not cofinal in $I$.
	\end{enumerate}
\end{defn}

\begin{thm}\label{equiv-sr}Let $\mc{M} \models \PA$ and $I \e M$. Then the following are equivalent:
	\begin{enumerate}
		\item \label{schmerl+omega} $I$ is $\omega$-coded and separable.
		\item \label{sr} $I$ is $0$-superrational.
	\end{enumerate}
\end{thm}

\begin{proof}
	$\eqref{schmerl+omega} \implies \eqref{sr}$: Suppose $I$ is $\omega$-coded, and let $a$ be such that $\sup(\{ (a)_i : i \in \omega \}) = I$ (the case in which $I$ is coded by $\omega$ from above is similar). Suppose also that $b \in M$ is such that $\Def_0(b) \setminus I$ is coinitial in $M \setminus I$. Then the following type is realized in $M$:
	\begin{align*}
	p(x) = &\{ (x)_{2n} = (a)_n : n \in \omega \} \\
	\cup &\{ (x)_{2n+1} = t_n(b) : n \in \omega \},
	\end{align*}
	where $\anglebracket{t_n : n \in \omega}$ is a recursive enumeration of all $\Delta_0$-definable Skolem functions. If $c$ realizes this type, then $\sup( \{ (c)_i : i \in \omega \} \cap I) = \inf(\{(c)_i : i \in \omega \} \setminus I) = I$, contradicting $\eqref{schmerl+omega}$.
	
	$\eqref{sr} \implies \eqref{schmerl+omega}$: \cite[Proposition 6.2]{kossak-omega} implies that if $I$ is $0$-superrational, then $I$ is $\omega$-coded. To see separability, notice that by $0$-superrationality, if $\Def_0(a) \cap I$ is cofinal in $I$, then $\Def_0(a) \setminus I$ is not coinitial in $M \setminus I$ (and vice versa).
\end{proof}

\cite[Theorem 6.5]{kossak-omega} states that $\omega$ is a strong cut if and only if every $\omega$-coded cut is $0$-superrational. Taken together with the above result, we see that if $\omega$ is not strong, then separable cuts are never $\omega$-coded. 

\begin{prop}\label{wsr-vs-least-z}For any $\mc{M} \models \PA$:
\begin{enumerate}
    \item\label{strength-least-z} If $\omega$ is a strong cut, then every cut $I$ which is $\omega$-coded is separable.
    \item\label{weak-least-z} If $\omega$ is not a strong cut, then every cut $I$ which is $\omega$-coded is not separable.
\end{enumerate}
\end{prop}

\begin{proof}
$\eqref{strength-least-z}$ is due to \cite[Theorem 6.5 $(a) \implies (c)$]{kossak-omega}. We show $\eqref{weak-least-z}$. Suppose $\omega$ is not strong. There is $a$ such that $\inf(\{ (a)_i : i \in \omega \} \setminus \omega) = \sup(\{ (a)_i : i \in \omega \} \cap \omega) = \omega$.

If $I \e M$ is a cut which is $\omega$-coded from above, then there is $c > I$ such that $I = \inf(\{ (c)_n : n \in \omega \})$. For simplicity assume that the sequence coded by $c$ is a strictly decreasing and its domain consists of all elements smaller than a nonstandard element $d$. Let $b$ code the sequence defined by $(b)_i = (c)_{(a)_i}$. We claim that $b$ witnesses the failure of separability of $I$.

Indeed, $(c)_{(a)_i} \in I$ if and only if $(c)_{(a)_i} < (c)_n$ for each standard $n$ if and only if $(a)_i > \omega$. Since the set $\{ (a)_i : i \in \omega \} \setminus \omega$ is coinitial with $\omega$, then  $\{ (c)_{(a)_i} : i \in \omega \} \cap I$ is cofinal with $I$. Indeed, for any $x\in I$ there is a nonstandard number $y<d$ such that $x<(c)_y\in I$. However, by the proerties of $a$ there is also a standard number $i\in \omega$ such that $\omega<(a)_i<y$. Since $c$ is strictly decreasing, it follows that for any such $i$, $x<(c)_{(a)_i}\in I.$ Similarly, since $\{ (a)_i : i \in \omega \} \cap \omega$ is cofinal with $\omega$, then $\{ (c)_{(a)_i} : i \in \omega \} \setminus I$ is coinitial with $I$.

The case when $I$ is upward $\omega$-coded is treated similarly.
\end{proof}

\begin{cor}\label{i(x)-not-sep}
Suppose $\mc{M} \models \PA$ is countable, recursively saturated but not arithmetically saturated. Then there are separable sets $X$ such that $I(X)$ is not separable.
\end{cor}

\begin{proof}
Let $c$ be nonstandard, and $I = \sup(\{ c + n : n \in \omega \})$. Then $I$ has no least $\mathbb{Z}$-gap above it, and so by Proposition \ref{no-least-z-0=1}, there is $S$ such that $(\mc{M}, S) \models \csm$ and $I = I(\IDC^{0=1}_S)$. Let $X = \IDC^{0=1}_S$. Then $X$ is separable by Theorem \ref{separability-theorem} and $I = I(X)$. Since $I$ is $\omega$-coded, by Proposition \ref{wsr-vs-least-z} $\eqref{weak-least-z}$, if $\omega$ is not a strong cut, then $I$ cannot be separable.
\end{proof}

Separable cuts always exist in recursively saturated models. We can in fact see more: every recursively saturated model $\mc{M}$ has a separable cut $I \e M$ which is not closed under addition. Moreover, $\mc{M}$ has separable cuts $I \e M$ that are closed under addition but not multiplication, and ones closed under multiplication but not exponentiation.

To see this, first notice that if $(\mc{M}, I)$ is recursively saturated and $I \e M$, then $I$ is separable. This follows directly from the equivalent definition of separability that says that for each $a$ there is $d$ such that for all $i \in \omega$, $(a)_i \in I$ iff $(a)_i < d$. Now let $I \e M$ be any cut not closed under addition. By resplendence, there is $J \e M$ such that $(\mc{M}, J)$ is recursively saturated and not closed under addition.

Again, notice that this proof generalizes to show that if $f$ and $g$ are increasing definable functions such that there is any cut $I \e M$ closed under $f$ but not $g$, then there is $I \e M$ that is separable and closed under $f$ but not $g$. Hence there are separable cuts which are closed under addition but not multiplication, and cuts which are closed under multiplication but not exponentiation.

\subsection{Arithmetic Saturation}

In \cite[Lemma 26]{cieslinski-lelyk-wcislo-dc}, we see that there exist \emph{disjunctively trivial} models: models $(\mc{M}, T) \models \ctm$ such that for all sequences $\anglebracket{\phi_i : i < c}$ of sentences such that $c > \omega$, $(\mc{M}, T) \models T(\bigvee\limits_{i < c} \phi_i)$. That is, models such that all disjunctions of nonstandard length are evaluated as true. In this part we see that disjunctive triviality implies arithmetic saturation.

\begin{defn}Let $(\mc{M}, S) \models \csm$ and $I \e M$. 
	\begin{enumerate}
		\item If, for every $c > I$ and every sequence of sentences (in the sense of $\mc{M}$) $\anglebracket{\phi_i : i < c}$, $(\mc{M}, S) \models T(\bigvee\limits_{i < c} \phi_i)$, then we say $(\mc{M}, S)$ is \emph{disjunctively trivial above $I$}. If $(\mc{M}, S)$ is disjunctively trivial above $\omega$, we simply say $(\mc{M}, S)$ is \emph{disjunctively trivial}.
		\item If, for every $c \in I$ and every sequence of sentences (in the sense of $\mc{M}$) $\anglebracket{\phi_i : i < c}$, $(\mc{M}, S) \models T(\bigvee\limits_{i < c} \phi_i) \equiv \exists i < c \: T(\phi_i)$, we say that $(\mc{M}, S)$ is \emph{disjunctively correct on $I$}.
	\end{enumerate}
\end{defn}

\begin{cor}\label{arith-sat}Suppose $(\mc{M}, S) \models \csm$ and $I \e M$. If $(\mc{M}, S)$ is disjunctively trivial above $I$ and disjunctively correct on $I$, then $I$ is separable. In particular, if $(\mc{M}, S)$ is disjunctively trivial above $\omega$, then $\mc{M}$ is arithmetically saturated. Conversely, if $\mc{M}$ is arithmetically saturated, there is $S$ such that $(\mc{M}, S) \models \csm$ and is disjunctively trivial above $\omega$.
\end{cor}

\begin{proof}
If $(\mc{M}, S) \models \csm$ is disjunctively trivial above $I$ and correct on $I$, then $I = \{ c : (\mc{M}, S) \models \lnot T(\bigvee\limits_c (0 = 1)) \}$. Therefore $I$ is separable by Theorem \ref{separability-theorem}. If $I = \omega$, then (by Proposition \ref{sep-iff-wsr}) $\omega$ is a strong cut in $\mc{M}$ and therefore $\mc{M}$ is arithmetically saturated. Conversely, suppose $\mc{M}$ is arithmetically saturated. We construct sequences $F_0 \subseteq F_1 \ldots$ of finitely generated sets of formulas such that $\cup F_i = \Form^{\mc{M}}$ and full satisfaction classes $S_0, S_1, \ldots$. Suppose $S_i$ is a full satisfaction class such that $(\mc{M}, S_i)$ is recursively saturated and if $\phi \in F_i \cap \Sent^{\mc{M}}$ is disjunction of nonstandard length, then $S_i(\phi, \emptyset)$. 

Let $a$ code the lengths of all disjunctions in $F_{i+1}$. That is, suppose $(b)_n$ is the $n$-th element of $F_{i+1}$, and $(a)_n$ is the maximum $c$ such that there is a sequence $\anglebracket{\phi_j : j < c}$ such that $(b)_n = \bigvee\limits_{j < c} \phi_j$. Since $\omega$ is strong, there is $d > \omega$ such that for each $n \in \omega$, $(a)_n \in \omega$ if and only if $(a)_n < d$. By \cite[Lemma 26]{cieslinski-lelyk-wcislo-dc}, the theory $Th$ asserting the following is consistent:
\begin{itemize}
    \item $\textnormal{ElDiag}(\mc{M})$,
    \item $S_{i+1}$ is compositional for each $\phi \in F_{i+1}$,
    \item $\{ S_i(\phi, \alpha) \equiv S_{i+1}(\phi, \alpha) : \phi \in F_i \}$ for all assignments $\alpha$ of $\phi$, and
    \item $\{ S_{i+1}(\bigvee\limits_{j < c} \phi_j, \alpha) : \bigvee\limits_{j < c} \phi_j \in F_{i+1}$ and $c > d \}$ for all assignments $\alpha$ of $\phi$.
\end{itemize}
Since $Th$ is a consistent, recursive theory and $(\mc{M}, S_i)$ is recursively saturated, by resplendence, $(\mc{M}, S_i)$ has an expansion $(\mc{M}, S_i, S_{i+1}) \models Th$. Continue as before, obtaining $S = \cup S_i \restriction F_i$, a full satisfaction class which is disjunctively trivial.
\end{proof}


We observe that, for each $n$, there is an arithmetical sentence $\theta_n$:= ``There exists a $\Delta_n$ full model of $\csm$ which is disjunctively trivial above $\omega$". Here by "$\omega$" we mean the image of the canonical embedding of the ground model onto an initial segment of the model and a "full model" means a model with a satisfaction relation satisfying the usual Tarski's truth condition. Corollary below shows that each such sentence is false.

\begin{cor}
For every $n$, $\mathbb{N}\models \neg\theta_n$.
\end{cor}
\begin{proof}
    Assume to the contrary and fix $n$ such that $\mathbb{N}\models \theta_n$. Fix a $\Delta_n$-definable model $\mathcal{M}:=(M,+,\cdot,S)\models \csm$ such that $\mathbb{N}\subseteq (M,+,\cdot)$ and $\mathcal{M}$ is disjunctively trivial above $\omega$. Then $(M,+,\cdot)$ is arithmetically saturated and consequently $(\mathbb{N}, \SSy(\mathcal{M}))\models \ACA_0$. However, each set in $\SSy(\mathcal{M})$ is $\Delta_n$ definable in $\mathbb{N}$, which is not possible.
\end{proof}

It follows from the above corollary that the construction of the disjunctively trivial model of $\ctm$ does not formalize in any true arithmetical theory, in particular it does not formalize in $\PA$.  Hence one cannot hope to interpret $\ctm+\textnormal{DC}-\textnormal{in}$ in $\PA$ by using the construction of 
a disjunctively trivial model internally in $\PA$. This is unlike in the case of a standard Enayat-Visser construction: \cite{elw} shows how to formalize the model theoretical argument from \cite{enayat-visser} in $\PA$ in order to conclude that $\ctm$ is feasibly reducible to $\PA$ and, in consequence, it does not have speed-up over $\PA$.


\section{Non-local Pathologies}\label{nonlocal-section}
In previous sections, we have considered a single, fixed $\theta$ and functions $F$ such that $F(x)$ is the $x$-th iterate of $\theta$ in some sense. We described sets defined by certain correctness properties with respect to this $\theta$. In other words, we explored ``local" pathologies (pathologies that are local to a fixed $\theta$). In this section we address several sets defined using non-local pathologies: for example, instead of fixing a $\theta$ and looking at the idempotent disjunctions of $\theta$, we look at all idempotent disjunctions (of any sentence). These sets can include $\IDC_S$, $\QC_S$, $\IDC^{\text{bin}}_S$, $\DNC_S$, among others.

\begin{rem}
Let us fix a model $(\mc{M}, S) \models \csm$ and consider $$\QC_S = \{ c : \forall \phi \in \Sent^{\mc{M}} T( (\forall x)^c \phi) \equiv T(\phi) \}.$$ Notice that $\QC_S$ is necessarily closed under addition, since if, for each $\phi$, $$T( (\forall x)^c \phi) \equiv T(\phi),$$ then let $\theta = (\forall x)^c \phi$, and so $$T( (\forall x)^c \theta) \equiv T(\theta) = T((\forall x)^c \phi) \equiv T(\phi).$$ Since $(\forall x)^c \theta = (\forall x)^{2c} \phi$, we conclude that $c \in \QC_S$ if and only if $2c \in \QC_S$. Suppose that $\QC_S$ is not a cut, and let $c_0 < c_1$ be such that $c_0 \notin \QC_S$ and $c_1 \in \QC_S$. Then there is $\phi$ such that $\lnot [T((\forall x)^{c_0} \phi) \equiv T(\phi)]$, but $T((\forall x)^{c_1} \phi) \equiv T(\phi)$. Then $c_1 \in \QC_S$, $2c_1 \in \QC_S$, but $c_0 + c_1 \notin \QC_S$, since $T( (\forall x)^{c_0 + c_1} \phi) \equiv T( (\forall x)^{c_0} \phi)$.

Let $I \e J_0 \e J_1 \e M$ be separable cuts closed under addition such that $c_0 \in J_0$ and $c_1 \in J_1 \setminus J_0$. Then $X = I \cup (J_1 \setminus J_0)$ is separable, but by the above argument, there can be no $S$ such that $(\mc{M}, S) \models \csm$ and $\QC_S = X$.
\end{rem}

This remark shows that there are complications that occur with sets defined by these non-local pathologies. For the remainder of this section, we look instead at the \emph{cuts} defined by these pathologies. 

We again generalize the setting to draw conclusions 
about $I(\IDC_S), I(\QC_S)$ and $I(\IDC^{\text{bin}}_S)$. To formalize this notion, we again look at finite propositional templates $\Phi(p, q)$ (recall this notion from the beginning of Section \ref{sep-section}). We restrict our attention to $\Phi$ with the following properties:
\begin{itemize}
    \item $\Phi(p, q)$ is not equivalent to $p$,
    \item the complexity of $\Phi(p,q)$ is non-zero,
    \item $q$ \textbf{must} appear in $\Phi(p, q)$,
    \item $p \wedge q \vdash \Phi(p, q)$, and
    \item $(\lnot p \wedge \lnot q) \vdash \lnot \Phi(p, q)$.
\end{itemize}
\begin{defn}
Suppose $\Phi$ has the above properties. Then $F : M \times \Sent \to \Sent$ defined as follows:
\begin{itemize}
    \item $F(0, \phi) = \phi$ for all $\phi \in \Sent^{\mc{M}}$, and
    \item $F(x + 1, \phi) = \Phi(\phi, F(x, \phi))$.
\end{itemize}
is called an \emph{idempotent sentential operator}. We say that $\Phi$ is a \emph{template for $F$}.
\end{defn}

Notice that for any $\theta$, the function $F(\cdot, \theta)$ is one to one.

\begin{lem}\label{accessible-vs-additive}
Let $\Phi$ be a template for $F$, and $F$ an idempotent sentential operator. If $p$ does not appear in $\Phi(p, q)$, then for all $x, y \in M$ and $\phi \in \Sent^{\mc{M}}$, $\mc{M} \models F(x + y, \phi) = F(x, F(y, \phi))$.
\end{lem}

\begin{proof}
If $p$ does not appear in $\Phi(p, q)$, then there is a propositional function $\Psi(q)$ such that $\Phi(p, q) = \Psi(q)$. Let $G : \Sent^{\mc{M}} \to \Sent^{\mc{M}}$ be defined by $G(\phi) = \Psi(\phi)$. Then, $$F(x + 1, \phi) = \Psi(F(x, \phi)) = G(F(x, \phi)).$$ Since $F$ and $G$ are $\mc{M}$-definable, by induction, one observes that for all $x$, $F(x, \phi) = G^x(\phi)$, the $x$-th iterate of $G$. Therefore, $$F(x + y, \phi) = G^{x + y}(\phi) = G^x(G^y(\phi)) = G^x(F(y, \phi)) = F(x, F(y, \phi)).\mbox{\qedhere}$$
\end{proof}

As before, notice that if $p$ appears in $\Phi(p, q)$, then for each $\phi$ and $x$, $\phi \in \Cl(F(x, \phi))$. For this reason, if $p$ appears in $\Phi(p, q)$, we refer to $F$ as \emph{accessible}. If not, then because of Lemma \ref{accessible-vs-additive}, we say $F$ is \emph{additive}.

\begin{defn}
Let $F$ be an idempotent sentential operator.
\begin{itemize}
    \item $\theta$ is \emph{$F$-irreducible} if whenever $F(x, \phi) = \theta$, then $\phi = \theta$ and $x = 0$.
    \item The \emph{$F$-length} of $\phi$ is the maximum $x$ such that there is $\theta$ such that $F(x, \theta) = \phi$.
    \item The \emph{$F$-root} of $\phi$ is the unique $\theta$ such that $F(x, \theta) = \phi$, where $x$ is the $F$-length of $\phi$.
\end{itemize}
\end{defn}

\begin{rem}
By working through the possible truth tables for $\Phi(p, q)$, one notices that if $\Phi(p, q)$ has the required properties, then it is logically equivalent to one of the following propositional formulae:
\begin{itemize}
    \item $p \vee q$,
    \item $p \wedge q$, or
    \item $q$.
\end{itemize} We say that $\Phi(p, q)$ is \emph{$q$-monotone} if it is logically equivalent to either $p \vee q$ or to $q$.

Notice that if $\phi \in \Sent^{\mc{M}}$, then in each of these cases, one can show that if $(\mc{M}, S) \models \csm$, $(\mc{M}, S) \models \forall x T(F(x, \phi)) \equiv T(F(x + 1, \phi))$.
\end{rem}

\begin{lem}\label{roots}Let $F$ be an idempotent sentential operator. 
\begin{enumerate}
    \item\label{accessible}If $F$ is accessible, then for all $x, y > 0$, $\phi, \psi \in \Sent^{\mc{M}}$, if $F(x, \phi) = F(y, \psi)$, then $x = y$ and $\phi = \psi$. In other words, when $x > 0$, the $F$-root of $F(x, \phi)$ is $\phi$.
    \item\label{additive}If $F$ is additive, then the $F$-root of $\phi$ is $F$-irreducible. Moreover, for all $x, y > 0$, $\phi, \psi \in \Sent^{\mc{M}}$, if $F(x, \phi) = F(y, \psi)$, then the $F$-root of $\phi$ and $F$-root of $\psi$ are the same. 
\end{enumerate}
\end{lem}

\begin{proof}
First we show $\eqref{accessible}$. Suppose $F$ is accessible and $F(x, \phi) = F(y, \psi)$. If $x, y > 0$, then $F(x, \phi) = \Phi(\phi, F(x - 1, \phi))$, and $F(y, \psi) = \Phi(\psi, F(y - 1, \psi))$. Since $F$ is accessible, then $p$ appears as a leaf of the syntax tree of $\Phi(p, q)$. Since $\Phi(\phi, F(x - 1, \phi)) = \Phi(\psi, F(y - 1, \psi))$, we see that $\phi = \psi$. One shows by induction (in $\mc{M}$, since $F$ is $\mc{M}$-definable) that if $F(x, \phi) = F(y, \phi)$, then $x = y$.

Next we show $\eqref{additive}$. Suppose $F$ is additive and $\theta$ is the $F$-root of $\phi$. Then $F(x, \theta) = \phi$ and $x$ is the $F$-length of $\phi$. If $\theta$ is not $F$-irreducible, then there is $y > 0$ and $\psi$ such that $F(y, \psi) = \theta$. Then $$\phi = F(x, \theta) = F(x, F(y, \psi)) = F(x + y, \psi),$$ the last equality holding by additivity. Since $x + y > x$, this contradicts that $x$ is the $F$-length of $\phi$.

To show the ``moreover" part of $\eqref{additive}$, let $x, y > 0$, $\phi, \psi \in \Sent^{\mc{M}}$, and $F(x, \phi) = F(y, \psi)$. Define $G : \Sent^{\mc{M}} \to \Sent^{\mc{M}}$ by $G(\phi) = \Phi(\phi, \phi)$, so that $F(x, \phi) = G^x(\phi)$. Notice that $G$ is one to one. Since $G$ is one to one, then if $x = y$, $G^x(\phi) = G^y(\psi)$ implies, by induction in $\mc{M}$, that $\phi = \psi$. Suppose $x > y$. Then again by induction in $\mc{M}$, we have that $\mc{M} \models G^{x - y}(\phi) = \psi$. Let $\theta$ be the $F$-root of $\phi$, so that there is $a$ such that $G^a(\theta) = \phi$. Then $G^{a + (x - y)}(\theta) = \psi$, so $\theta$ is the $F$-root of $\psi$. 

\end{proof}

Consider the following examples of $\Phi(p, q)$:
\begin{itemize}
    \item $\Phi(p, q) = q \vee p$. In this case, $F(x, \phi) = \bigvee\limits_x \phi$.
    \item $\Phi(p, q) = q \wedge p$. In this case, $F(x, \phi) = \bigwedge\limits_x \phi$.
    \item $\Phi(p, q) = (\forall y) q$. Then $F(x, \phi) = \underbrace{\forall y \ldots \forall y}_{ x \textnormal{ times}} \phi$.
    \item $\Phi(p, q) = q \vee q$. Then $F(x, \phi) = \bigvee\limits^{\text{bin}}_x \phi$.
    \item $\Phi(p, q) = \lnot \lnot q$. Then $F(x, \phi) = (\lnot \lnot)^x \phi$.
\end{itemize}
The goal of this section is to characterize those cuts $I$ such that $$I = \{ c : \forall x \leq c \forall \phi \in \Sent^{\mc{M}} T(\phi) \equiv T(F(x, \phi)) \}.$$ This would allow us to characterize $I(\IDC_S)$, $I(\IDC^{\text{bin}}_S)$, and $I(\QC_S)$, among others. For $\IDC^{\text{bin}}_S$ and $\QC_S$ the relevant $F$ functions are additive, while for $\IDC_S$, $F$ is accessible. For the most part we will restrict our attention to $\Phi$ with syntactic depth 1. This covers most of the above cases, with the notable exception of $\lnot \lnot q$; we treat this case separately.

\begin{thm}
Let $(\mc{M}, S) \models \csm$ and suppose $F$ is an additive idempotent sentential operator. If $$I = \{ c : \forall x \leq c \forall \phi \in \Sent^{\mc{M}} T(\phi) \equiv T(F(x, \phi)) \},$$ then $I$ is closed under addition.
\end{thm}

\begin{proof}
Let $a \in I$. We show $2a \in I$. To see this, let $x \leq 2a$. If $x \leq a$, we are done. Otherwise, let $b = x - a$, so $x = a + b$ and $a, b \leq a$. Notice that for $\phi \in \Sent^{\mc{M}}$, we have $(\mc{M}, S) \models T(\phi) \equiv T(F(a, \phi))$ and $(\mc{M}, S) \models T(F(a, \phi)) \equiv T(F(b, F(a, \phi))$. By additivity, $F(b, F(a, \phi)) = F(a + b, \phi)$, and $x = a + b$, so we have $$(\mc{M}, S) \models T(\phi) \equiv T(F(x, \phi)).\qedhere$$
\end{proof}

Given a cut $I \e M$, we say $I$ is \emph{$F$-closed} if either $F$ is accessible or $F$ is additive and $I$ is closed under addition. We say $I$ \emph{has no least $F$-gap} if one of the following holds:
\begin{itemize}
    \item $F$ is accessible and if $x > I$, then there is a $y$ such that for each $n\in\omega$, $x - n >y> I$, or
    \item $F$ is additive and if $x > I$, there is a $y$ such that for each $n\in\omega$, $\lfloor \frac{x}{n} \rfloor > y>I$.
\end{itemize}

Our next main result shows that if $I$ is $F$ closed and either separable or has no least $F$-gap, then there is $S$ such that $(\mc{M}, S) \models \csm$ and $$I = \{ c : \forall x \leq c \forall \phi \in \Sent^{\mc{M}} T(\phi) \equiv T(F(x, \phi)) \}.$$ Our method of proof will be similar to our previous results: we build sequences of finitely generated sets $F_0 \subseteq F_1 \subseteq \ldots$ and full satisfaction classes $S_0, S_1, \ldots$ with particular properties. We first prove two important lemmas which we use in the inductive step of our construction.

For the rest of this section, we modify Definition \ref{sep} so that we say $\phi \triangleleft \psi$ if either $\phi$ is an immediate subformula of $\psi$ or $\phi$ is the $F$-root of $\psi$. Similarly modify the definitions of closed sets and finitely generated sets so that such sets are closed under $F$-roots. Note that by Lemma \ref{roots}, if $F$ is accessible, this changes nothing about finitely generated and/or closed sets, but this does have an effect for additive $F$.

\begin{defn}
Let $F$ be an idempotent sentential operator with template $\Phi(p, q)$. Let $I \e M$, $X \subseteq \Form^{\mc{M}}$ closed, and $S$ a full satisfaction class.
\begin{enumerate}
    \item $S$ is \emph{$F$-correct on $I$} for formulae in $X$ if for each $\phi \in X$ and $x \in M$, whenever $F(x, \phi) \in X$ and $x \in I$, then $S(F(x, \phi), \alpha)$ if and only if $S(\phi, \alpha)$ for all assignments $\alpha$ of $\phi$.
    \item $S$ is \emph{$F$-trivial above $I$} for formulae in $X$ if for each $\phi \in X$ and $x \in M$, whenever $F(x, \phi) \in X$ and $x > I$, then either $\Phi(p, q)$ is $q$-monotone and $S(F(x, \phi), \alpha)$ for all assignments $\alpha$, or $\Phi(p, q)$ is not $q$-monotone and $\lnot S(F(x, \phi), \alpha)$ for all assignments $\alpha$ of $\phi$.
\end{enumerate}
\end{defn}

\begin{lem}\label{nonlocal-wsr-lemma}
Let $\mc{M} \models \PA$ be countable and recursively saturated. Let $F$ be an idempotent sentential operator with template $\Phi(p, q)$, and assume $\Phi(p, q)$ has syntactic depth 1. Let $I \e M$ be $F$-closed and separable. Suppose $X \subseteq \Form^{\mc{M}}$ is finitely generated, $S$ is a full satisfaction class, $(\mc{M}, S)$ is recursively saturated, $S$ is $F$-correct on $I$ for sentences in $X$ and $F$-trivial above $I$ for sentences in $X$. Then for any finitely generated $X^\prime \supset X$, there is a full satisfaction class $S^\prime$ such that $(\mc{M}, S^\prime)$ is recursively saturated, $S^\prime \restriction X = S \restriction X$, $S^\prime$ is $F$-correct on $I$ for sentences in $X^\prime$, and $S^\prime$ is $F$-trivial above $I$ for sentences in $X^\prime$.
\end{lem}

\begin{proof}
Suppose we have $X, X^\prime$ and $S$ as given in the statement of the Lemma. Let $a, b, c$ code enumerations such that $\{ (c)_n : n \in \omega \}$ enumerates $\Sent^{\mc{M}} \cap X^\prime$, $(b)_n$ is the $F$-root of $(c)_n$, and $F((a)_n, (b)_n) = (c)_n$. By separability, there is $d$ such that for each $n \in \omega$, $(a)_n \in I$ if and only if $(a)_n < d$.

To obtain $S^\prime$, we proceed in two stages. First, we find an $X^\prime$-satisfaction class $S_1$ with the following properties:
\begin{itemize}
	\item $(\mc{M}, S_1)$ is recursively saturated,
	\item $S_1 \restriction X = S \restriction X$,
	\item if $\psi = F(x, \phi) \in X^\prime$, $\phi \in X^\prime$ and $x \in I$, then  $(\psi, \emptyset) \in S_1$ iff $(\phi, \emptyset) \in S_1$, and 
    \item if $\psi = F(x, \phi) \in X^\prime$, and $x > I$, then if $\Phi(p, q)$ is $q$-monotone, then $(\psi, \emptyset) \in S_1$, and if $\Phi(p, q)$ is not $q$-monotone, then $(\psi, \emptyset) \not \in S_1$.
\end{itemize} As before, by \cite[Lemma 3.1]{enayat-visser}, $\mc{M}$ has an elementary extension $\mc{N}$ with a full satisfaction class $S_N$ agreeing with $S_1$ on $X^\prime$. Moreover, if $U_c = \{ x : \anglebracket{c, x} \in S_1 \}$ for $c \in X^\prime$, then $(\mc{M}, U_c)_{c \in X^\prime}$ is  resplendent, so we can assume that there is a full satisfaction class $S^\prime$ on $\mc{M}$ itself and that $(\mc{M}, S^\prime])$ is recursively saturated.

To find such an $X^\prime$-satisfaction class $S_1$, we use an Enayat-Visser argument. Let $Th$ be the theory in the language $\lpa \cup \{ S, S_1 \}$ asserting the following:

\begin{enumerate}
	\item the elementary diagram of $\mc{M}$,
	\item $\Comp(\phi,\psi,\theta)[S_1/S]$, for all $\phi,\psi,\theta \in X^\prime$,
	\item $\{ S_1(\phi, \alpha) \equiv S(\phi, \alpha) : \phi \in X \}$ (preservation),
	\item $\{ S_1(F((a)_n, (b)_n), \emptyset) \equiv S_1((b)_n, \emptyset ) : n \in \omega, (a)_n < d \}$ ($F$-correctness), and one of the following:
	\item\label{triv} \begin{itemize}
	    \item if $\Phi(p, q)$ is $q$-monotone, then $$\{ S_1(F((a)_n, (b)_n), \emptyset ) : n \in \omega, (a)_n > d \},$$ or
	    \item if $\Phi(p, q)$ is not $q$-monotone, then $$\{ \lnot S_1(F((a)_n, (b)_n), \emptyset ) : n \in \omega, (a)_n > d \}.$$
	\end{itemize}
\end{enumerate}
We refer to either of the conditions in $\eqref{triv}$ as $F$-triviality.

Before we show that $Th$ is consistent, we first show that if $(\mc{M}, S, S_1) \models Th$ is recursively saturated, then $S_1$ satisfies the required properties. First, since $S_1$ is compositional for formulas in $X^\prime$, $S_1$ is an $X^\prime$-satisfaction class. Moreover, the preservation scheme implies that $S_1\restriction X = S \restriction X$. To show the other properties, suppose $\theta = F(x, \phi) \in X^\prime$ and $\phi \in X^\prime$. Then there are $n, m \in \omega$ such that $\theta = F((a)_n, (b)_n)$ and $\phi = F((a)_m, (b)_m)$. By Lemma \ref{roots}, if $F$ is accessible, then either $x = 0$ and $\theta = \phi$, or $x = (a)_n$ and $\phi = (b)_n$; so if $F$ is accessible, there is nothing to show. Suppose $F$ is additive. Therefore (by our hypothesis) $I$ is closed under addition. By Lemma \ref{roots}, $(b)_n = (b)_m$ and $(a)_n = (a)_m + x$. There are two cases to consider, corresponding to the $F$-correctness and $F$-triviality properties of $\theta$:

\noindent \underline{Case 1}: $x \in I$ ($F$-correctness): Since $I$ is closed under addition, $(a)_n \in I$ if and only if $(a)_m \in I$. By separability, therefore, $(a)_n < d$ if and only if $(a)_m < d$. If $(a)_n < d$, then by $F$-correctness we have $(\theta, \emptyset) \in S_1$ if and only if $ ((b)_n, \emptyset) \in S_1$ and $(\phi, \emptyset) \in S_1$ if and only if $((b)_n, \emptyset) \in S_1$. Therefore, $(\theta, \emptyset) \in S_1$ if and only if $(\phi, \emptyset) \in S_1$. If $(a)_n > d$, then by $F$-triviality we have either $(\theta, \emptyset) \in S_1$ and $(\phi, \emptyset) \in S_1$, or $(\theta, \emptyset) \notin S_1$ and $(\phi, \emptyset) \notin S_1$. Again we have $(\theta, \emptyset) \in S_1$ if and only if $(\phi, \emptyset) \in S_1$.

\noindent \underline{Case 2}: $x > I$ ($F$-triviality): In this case, $(a)_n > I$, and therefore $(a)_n > d$. By $F$-triviality, if $\Phi$ is $q$-monotone, we have $(\theta, \emptyset) \in S_1$, and if $\Phi$ is not $q$-monotone, we have $(\theta, \emptyset) \notin S_1$.

Now we return to showing that $Th$ is consistent. Let $T_0 \subseteq Th$ be a finite subtheory. Let $C$ be the set of formulas such that the instances of their compositionality, preservation, $F$-correctness and $F$-triviality appear in $T_0$. Then $C$ is finite, so the modified subformula relation, $\triangleleft$, is well-founded on $C$. We define $S$ inductively on this relation:

Suppose $\phi$ is minimal in $C$. If $\alpha$ is an assignment for $\phi$, we put $(\phi, \alpha) \in S_1$ if any of the following hold: 
\begin{enumerate}
	\item $\phi \in X$ and $(\phi, \alpha) \in S$,
        \item $\phi$ is atomic, $\alpha$ is an assignment for $\phi$ and $\mc{M} \models \phi[\alpha]$, or
	\item $\Phi(p, q)$ is $q$-monotone, $\phi = F((a)_n, (b)_n)$, $\alpha = \emptyset$ and $(a)_n > d$.
\end{enumerate} 
Define $\phi$ of higher rank using compositionality if possible. If it is not possible, meaning that no immediate subformula of $\phi$ is in $C$, then there must be $\psi \in C$ such that $\psi$ is the $F$-root of $\phi$. Let $\phi = F((a)_n, (b)_n)$, where $(b)_n = \psi$. In this case, put $(\phi, \alpha) \in S_1$ if either $(\psi, \alpha) \in S_1$ or $(a)_n > d$ and $\Phi$ is $q$-monotone. 

We show that $(\mc{M}, S, S_1) \models T_0$. Clearly, $(\mc{M}, S, S_1)$ satisfies the elementary diagram of $\mc{M}$, and by definition, \allowbreak $(\mc{M}, S, S_1)$ satisfies all compositional axioms in $T_0$. 

We show that $(\mc{M}, S, S_1)$ satisfies the preservation scheme. Suppose $\phi \in X$. Then if $\phi$ is minimal in $C$ in the subformula relation, then $S_1(\phi, \alpha)$ if and only if $S(\phi, \alpha)$ by construction. If $\phi$ is not minimal, then $S_1(\phi, \alpha)$ if and only if $S(\phi, \alpha)$ follows by compositionality along with $F$-correctness and $F$-triviality of $S$ on sentences from $X$.

Next we show $F$-triviality. Suppose $\phi = F((a)_n, (b)_n) \in C$ and $(a)_n > d$. We assume $\Phi(p, q)$ is $q$-monotone; the other case is similar. If $\phi$ is minimal in $C$, then by construction, $(\phi, \emptyset) \in S_1$. If $\phi$ is not minimal, then let $\psi = F((a)_n - 1, (b)_n)$. As $(a)_n > I$, $(a)_n - 1 > I$ as well, so $(a)_n - 1 > d$. If $\psi \in C$, then by induction, we have $(\psi, \emptyset) \in S_1$. By compositionality, $(\psi, \emptyset) \in S_1$ if and ony if $(\phi, \emptyset) \in S_1$, so $(\phi, \emptyset) \in S_1$. If $\psi \not \in C$, then it must be the case that $(b)_n \in C$, and by construction, $(\phi, \emptyset) \in S_1$ since $(a)_n > d$.

Lastly, we show the $F$-correctness scheme. Suppose $\phi = F((a)_n, (b)_n) \in C$, $(a)_n < d$, and $S_1(\phi, \emptyset) \equiv S_1((b)_n, \emptyset) \in T_0$. If $\phi \in X$, then $(b)_n \in X$, and $(\phi, \emptyset) \in S$ if and only if $((b)_n, \emptyset) \in S$. By preservation, the same holds with $S_1$ replacing $S$. Suppose $\phi \not \in X$. Let $\psi = F((a)_n - 1, (b)_n)$. If $\psi \in C$, then as $\psi$ and $(b)_n$ each have lower rank than $\phi$, we can assume $((b)_n, \emptyset) \in S_1$ if and only if $(\psi, \emptyset) \in S_1$. Then by compositionality, we have $S(\phi, \alpha) \equiv S(\psi, \alpha)$, so, $$(\phi, \emptyset) \in S_1 \iff (\psi, \emptyset) \in S_1 \iff ((b)_n, \emptyset) \in S_1.$$ If $\psi \not \in C$, then by our construction, $(\phi, \emptyset) \in S_1$ if and only if either $((b)_n, \emptyset) \in S_1$ or $(a)_n > d$ (and $\Phi$ is $q$-monotone). Since $(a)_n < d$, then $(\phi, \emptyset) \in S_1$ if and only if $((b)_n, \emptyset) \in S_1$.

Since $Th$ is consistent, there is a model $(\mc{M}^\prime, S^\prime, S_1^\prime) \models Th$. By resplendency of $(\mc{M}, S)$, $(\mc{M}, S)$ has an expansion $(\mc{M}, S, S_1) \models Th$. This $S_1$ is the required $X^\prime$-satisfaction class.
\end{proof}  

We shall now prove an analogous lemma with a different assumption about $I$: instead of separability we shall require that there is no least $F$-gap above $I$. In the proof we shall need one more notion, which we shall now define:

\begin{defn}\label{def_d_closure}
Let $\mc{M}\models \PA$, and let $F$ be an idempotent sentential operator. Assume that $F$ is additive.
    For $Z\subseteq \Form^{\mc{M}}$ and $d\in M$, let $Z_{d}$ be the set of those formulae of the form $F(c,\phi)$, for which there are $n\in\mathbb{N}$, $a\in M$, such that 
    \begin{itemize}
        \item $0<a-c<n\cdot d$,
        \item $F(a,\phi)\in Z$,
        \item $\phi$ is the root of $F(a,\phi)$.
    \end{itemize}
    For uniformity of our proofs, when $F$ is accessible, we take $Z_d$ to be just the closure of $Z$ (under immediate subformulae and taking $F$-roots).
\end{defn}

\begin{prop}\label{prop_idempot_d_closure}
Let $\mc{M}\models \PA$, $F$ an idempotent sentential operator, and $Z\subseteq \Form^{\mc{M}}$. Then, for every $d\in M$ $(Z_d)_d\subseteq Z_d$.
\end{prop}
\begin{proof}
     This is clear if $F$ is accessible, so assume $F$ is additive. Fix an arbitrary $c,\phi$ such that $F(c,\phi)\in (Z_{d})_{d}$. Choose $a,n$ such that $F(a,\phi)\in Z_{d}$ and $0<a-c< n\cdot d$. By definition it follows that for some $c',$ $n'$,  $\phi'$ and $a'$, $F(a,\phi) = F(c',\phi')$, $F(a',\phi')\in Z$ and $0<a'-c'<n'\cdot d$. Since $F$ is additive this means that $\phi=\phi'$ (since both of them are roots) and $a=c'$, hence 
    \[0<a'-c = a' - a + a -c = a'-c'+a-c < (n+n')\cdot d.\]  
\end{proof}

\begin{lem}\label{nonlocal-fgap-lemma}
Let $\mc{M} \models \PA$ be countable and recursively saturated. Let $F$ be an idempotent sentential operator with template $\Phi(p, q)$, and assume $\Phi(p, q)$ has syntactic depth 1. Let $I \e M$ be $F$-closed and has no least $F$-gap. Suppose $S$ is a full satisfaction class, $(\mc{M}, S)$ is recursively saturated, $d > I$ and $S$ is $F$-correct on $[0, d).$ Suppose further that  $X \subseteq \Form^{\mc{M}}$ is finitely generated. Then for any formula $\tilde{\phi} \in \Form^{\mc{M}}$, there are $I < d_0 < d_1 < d$, a finitely generated $X^\prime \supseteq X$ and a full satisfaction class $S^\prime$ such that $\tilde{phi} \in X^\prime$, $(\mc{M}, S^\prime)$ is recursively saturated, $S^\prime \restriction X = S \restriction X$, $S^\prime$ is $F$-correct on $[0, d_0)$ and, for some $\theta \in X^\prime$, $F(d_1, \theta) \in X^\prime$ and  $(\mc{M}, S^\prime) \models \lnot (S^\prime(\theta, \emptyset) \equiv S^\prime(F(d_1, \theta), \emptyset))$. 
\end{lem}

\begin{proof}

Fix $\mc{M},I, S, X, d$ and $\tilde{\phi}$ as in the assumptions. Let  $\odot$ denote $+$ if $F$ is accessible and $\cdot$ if $F$ is additive.  Let $d_1$, $d_0$ be any numbers above $I$ such that for every $n,k\in\omega$, $d_0\odot n <d_1\odot k < d$. Suppose that every formula in $X\cup\{\tilde{\phi}\}$ has complexity smaller than $r\in M$. Let $\theta := (\neg)^{2r}0=0$ if $F$ is not $q$-monotone and $\theta:=\neg (\neg)^{2r}0=0$ in the other case. We note that $\theta$ is the $F$-root of $F(d_1, \theta)$ and $\Cl(F(d_1,\theta))$ is disjoint from $X$. We put $Y:= \Cl(X\cup \{F(d_1,\theta)\})$. Observe that if  $\phi\in Y$ is an $F$-root, then either $\phi \in X$ or $\phi  =\theta$.  Hence $Y$ is closed under $F$-roots.

 We shall start our construction by extending $S$ to a $Y\cup Y_{d_0}$-satisfaction class on $\mc{M}$ which if $F$-correct on $[0,d_0)$. Proposition \ref{prop_idempot_d_closure} implies that $(Y_{d_0})_{d_0}\subseteq Y_{d_0}$. Since obviously, for any $Z, Z'$ $(Z\cup Z')_{d_0} = Z_{d_0}\cup Z'_{d_0}$, it follows that $(Y\cup Y_{d_0})_{d_0} = Y_{d_0}\cup (Y_{d_0})_{d_0}\subseteq Y\cup Y_{d_0}.$ Additionally $Y\cup Y_{d_0}$ is closed under roots and under immediate subformulae. We argue that $X_{d_0}\cap \Cl(F(d_1,\theta))_{d_0} = \emptyset.$ To this end observe that if $\psi\in \Cl(F(d_1,\theta))_{d_0}$, then either $\psi$ is in $\Cl(\theta)$, and hence the complexity of $\psi$ is greater than $2r-n$ for some standard $n$, or $\psi$ contains $\theta$ as a subformula. In both cases the complexity of $\psi$ is at least $2r-n$ for some standard $n$. Consequently, if $\psi\in \Cl(F(d_1,\theta))_{d_0}$, then $\psi$ does not belong to $X_{d_0}$, because each formula in $X_{d_0}$ is a subformula of a formula in $X$, and hence its complexity is not greater than $r$.  Moreover, if $\phi$, $F(b,\phi)$ are both in $Y\cup Y_{d_0}$ and $b<d_0$, then $\phi \in X_{d_0}\iff F(b,\phi)\in X_{d_0}$. Indeed, from right to left  this follows since $(X_{d_0})_{d_0}\subseteq X_{d_0}$. From left to right this is so, since if $F(b,\phi)\notin X_{d_0}$, then either $F(b,\phi)\in \Cl(\theta)_{d_0}$ or $F(b,\phi) = F(b',\theta)$. The first case is impossible since each formula in $\Cl(\theta)_{d_0}$ starts with a negation which does not occur in $\Phi$. In the latter case it follows that $\theta$ is a subformula of $\phi$ (because $\theta$ is $F$-irreducible) and hence $\phi\notin X_{d_0}$. 

Let us put $Y' = Y\cup Y_{d_0}$. We extend $S\restr{X}$ to a $Y'$-satisfaction class $S_0$, which is compositional and $d_0$-correct for formulae in $Y'$. For every $\phi\in Y'$ and every $\alpha$:
\begin{itemize}
    \item if $\phi \in X$, then $S_0(\phi,\alpha)$ iff $S(\phi,\alpha)$;
    \item if $\phi \in\Cl(\{\theta\})$, then ($\phi$ is a sentence and)  $S_0(\phi,\alpha)$ iff $\alpha = \emptyset$ and $\phi$ is of the form $(\neg)^{2b}0=0$.
    \item if $\phi = F(d_1,\theta)$, then ($\phi$ is a sentence and) $S_0(\phi,\alpha)$ iff $\alpha = \emptyset$ and $F$ is $q$-monotone.
    \item if $\phi$ is in the closure of $F(d_1,\theta)$, then, since $\Phi(p,q)$ has syntactic depth $1$, $\phi$ is either in $\Cl(\{\theta\})$ or $\phi = F(d_1-n,\theta)$ for some $n\in\omega$. We have already taken care of the former case. In the latter case we let the value of $\phi$ on $\alpha$ be the same as that of $F(d_1,\theta)$ on $\alpha$.
    \item otherwise $\phi =  F(a - b,\psi)$ for some $k\in\mathbb{N}$, $a\in M$, $b<k\cdot d_0$  and $\psi$, $F(a,\psi)$ such that $F(a,\psi)\in Y$, $\psi$ is an $F$-root of $F(a,\psi)$. This can happen only if $F$ is additive. Since $Y$ is closed under roots, $\psi\in Y$, hence for each $\alpha$ the value of $\psi$ on $\alpha$ has already been defined. We stipulate that the value of $F(a-b,\psi)$ on $\alpha$ is the same as that of $F(a,\psi)$ on $\alpha$. We observe that this is independent of the choice of $F(a,\psi)\in Y$: if $F(a,\psi)$ and $F(a',\psi')$ both satisfy the above conditions, then either both $F(a,\psi), F(a',\psi')$ belong to $X$ or both of them belong to $\Cl(F(d_1,\theta))$. If the former holds our claim  follows because $S$ is $F$-correct on $[0,d)$. If the latter holds, it must be the case that $\psi = \psi' = \theta$ and $|a-a'|$ is standard, so our claim follows by construction.
    \end{itemize}
We check that $S_0$ is $F$-correct on $[0,d_0)$ for sentences in $Y'$. If $F$ is accessible, this easily follows from our construction. Assume that $F$ is additive. Assume $0<b<d_0$ and fix an arbitrary $\phi$. By previous considerations either both $\phi,F(b,\phi)$ belong to $X_{d_0}$ or they both belong to $\Cl(F(d_1,\theta))_{d_0}.$ In the latter case both $\phi$ and $F(b,\phi)$ are of the form $F(d_1-b',\theta)$ for $b'<n\cdot d_0$. In particular, for an arbitrary $\alpha$, $\phi$ and $F(b,\phi)$ get the same value on $\alpha$ (by construction).

 Suppose then that $\phi,F(b, \phi)\in X_{d_0}$ and fix $a_0,a_1,b_0,b_1,n_0, n_1,\psi_0,\psi_1$ such that $\phi = F(a_0-b_0,\psi_0)$, $F(b,\phi) = F(a_1-b_1,\psi_1)$ and $F(a_i,\psi)\in X$, $b_i<n_i\cdot d_0$ and $\psi_i$ is the root of $F(a_i,\psi_i)$. 
 It follows that $\phi$ and $F(b,\phi)$ have the same root, so $\psi_0 = \psi_1$. In particular $F(b,\phi) = F(a_1-b_1,\psi_0) = F(a_0-b_0+b,\psi_0)$. Hence $a_1-b_1 = a_0-b_0+b,$ so $|a_1-a_0| = |b_1+b - b_0|\leq 3d_0<d$. In particular, since $S$ is $F$-correct on $[0,d)$, $F(a_0,\psi_0)$ and $F(a_1,\psi_0)$ are assigned by $S$ the same values on each $\alpha$.

Now we show how to find $S'$ and $X'$ as in the statement of the lemma. We let $X' = X \cup \Cl(\{\tilde{\phi},\theta, F(d_1,\theta)\})$. For $S'$, by an easy resplendency argument, it is enough to build an extension $\mc{N}\succeq \mc{M}$ and a satisfaction class $S_N$ such that
\begin{enumerate}
    \item $S_N$ is an $\mathcal{N}$ satisfaction class which is $F$-correct satisfaction class on $[0,d_0)$.
    \item $S_N$ makes $F(d_1,\theta)\equiv \theta$ false.
    \item $S_N$ agrees with $S$ on $X$.
\end{enumerate}
We observe that since $X$ is finitely generated, then condition $3$ is expressible in the language of arithmetic augmented with $S_N$ and $S$. In the construction we shall heavily rely on the extension of $S$ to $Y'$ given by $S_0$. We build $\mc{N}$ and $S_N$ in stages following the idea of \cite{enayat-visser}. Let $\mc{M}_0 = \mc{M}$, and we construct a chain of pairs $(\mc{M}_0, S_0), (\mc{M}_1, S_1), \ldots$ which satisfy the following conditions
\begin{itemize}
    \item for each $n$, $\mc{M}_{n}\preceq \mc{M}_{n+1}$.
    \item for each $n$, $S_{n+1}$ is a $\Form^{\mc{M}_n}$-satisfaction class.
    \item $S_1$ agrees with $S_0$ on $Y^\prime$ and for each $n>1$, $S_{n+1}$ agrees with $S_n$ on $\Form^{\mc{M}_n}$.
    \item for each $n$, $S_{n+1}$ is $F$-correct on $[0,d_0)$ with respect to formulae from $\Form^{\mc{M}_n}$.
\end{itemize}

We show how to construct $\mc{M}_1, S_1$ and the construction of $\mc{M}_{n+1}, S_{n+1}$ given $\mc{M}_n, S_n$ for $n\geq 1$ is fully analogous. We consider the theory given as the union of the following sets of sentences:
\begin{itemize}
    \item[A] $\textnormal{ElDiag}(\mc{M}_0)$;
    \item[B] $\{S(\phi,\alpha) : \phi \in Y', (\phi, \alpha) \in S_0\}$
    \item[C] $\{ \textnormal{Comp}(\phi,\psi,\theta) : \phi\in \Form^{\mc{M}_0}\}.$
    \item[D] $\{\forall \alpha S(F(a,\phi)\equiv \phi, \alpha) : a<d_0, \phi\in\Form^{\mc{M}_0}\}.$
\end{itemize}

Fix a finite portion $T_0$ of this theory and let $E$ be the set of those formulae which occur in one of the axioms in $T_0$. 

Let us observe that the relation $\phi\sqsubset \psi := \mc{M}_0\models "\phi\textnormal{ is a subformula of } \psi$" is well-founded on $E$, since $E$ is finite. By this we mean that $\phi \sqsubset \psi$ if $\mc{M}_0$ sees that $\phi$ is a subformula (not necessarily direct) of $\psi$. We define $S\subseteq M_0^2$ by induction on the ranking function rk($\cdot$) given by $\sqsubset$. For an arbitrary $\psi$ of rank $0$ we put
\begin{itemize}
    \item if $\psi$ is standard, then we know what to do.
    \item if $\psi\in Y'$, then $(\psi,\alpha) \in S$ iff $(\psi,\alpha)\in S_0$
    \item if $\psi\notin Y'$, then for no $\alpha$, $(\psi,\alpha) \in S$.  
\end{itemize}
If $\phi$ has positive rank, then
\begin{itemize}
    \item if all immediate subformulae are in $E$, then the immediate subformulae of $\phi$ have lower ranks, so we know what to do.
    \item if the above does not hold and $\phi = F(a,\psi)$ for some $\psi\in E$ and $0<a<d_0$, then $\psi$ has lower rank, so for an arbitrary $\alpha$ we put $(\phi,\alpha) \in S$ iff $(\psi,\alpha) \in S$. 
     \item if $\psi\in Y'$, then $(\psi,\alpha) \in S$ iff $(\psi,\alpha) \in S_0$.
    \item otherwise, for every $\alpha, (\phi,\alpha) \notin S.$
    
\end{itemize}

We check that with so defined $S$, $(\mc{M},S)\models T_0.$ That the compositional clauses hold is clear from the construction. We check that $S$ is $F$-correct on $[0,d_0)$ for sentences in $E$. By induction on $n$ we prove that for all $\phi, F(a,\phi)\in E$, $\textnormal{rk}(\phi)+\textnormal{rk}(F(a,\phi)) = n$, $a < d_0$, then for every $\alpha$, $S(\phi,\alpha)\iff S(F(a,\phi),\alpha).$ Since $\textnormal{rk}(\phi)+\textnormal{rk}(F(a,\phi)) = 0$ only if $a=0$, the base step is trivial. Assume $\rk(\phi)+ \rk(F(a,\phi))$ is positive. Then certainly $\rk(F(a,\phi))$ is positive. If all immediate subformulae of $F(a,\phi)$ belong to $E$, then at least one of them is of the form $F(a-1,\phi)$ and the thesis follows by inductive hypothesis and idempotency of $\Phi$, since $F(a-1,\phi)$ has lower rank than  $F(a,\phi)$. Otherwise, for some $\psi\in E$ and $b<d_0$ such that $F(a,\phi) = F(b,\psi)$ and we decided that for every $\alpha$, the values of $F(b,\psi)$ and $\psi$ are the same. By Lemma \ref{roots}, for some $b'$, either $\phi = F(b',\psi)$ or $\psi = F(b',\phi)$. Hence the thesis follows by the inductive assumption.


Now we argue for the preservation axioms. By induction on the rank of $\phi$ we prove that if $\phi\in Y'$, then for every $\alpha$, $S(\phi,\alpha)$ iff $S_0(\phi,\alpha).$ This is immediate for formulae of rank $0$. In the induction step we use the definition of $S$ and the closure properties of $Y'$.

For the step induction step we first consider extend $S_n\upharpoonright_{M_n}$ to the set $\textrm{Form}^{\mc{M}}\cup (\textrm{Form}^{\mc{M}})_{d_0}\subseteq \textrm{Form}^{\mc{M}_{n+1}}.$ Then we argue as in the first step.
\end{proof}

\begin{thm}\label{nonlocal}Let $\mc{M} \models \PA$ be countable and recursively saturated and $I \e M$. Let $F$ be an idempotent sentential operator with template $\Phi(p, q)$, and assume $\Phi(p, q)$ has syntactic depth 1. Suppose $I$ is $F$-closed. Then if $I$ is separable or has no least $F$-gap above it, there is $S$ such that $(\mc{M}, S) \models \csm$ and $$I = \{ x : \forall y \leq x \forall \phi \in \Sent^{\mc{M}} (T(\phi) \equiv T(F(y, \phi))) \}.$$ \end{thm}

\begin{proof}
We construct sequences $F_0 \subseteq F_1 \subseteq \ldots$ and $S_0, S_1, \ldots$ of sets such that:
\begin{enumerate}
	\item $F_i$ is finitely generated and $\cup F_i = \Form^{\mc{M}}$,
	\item $S_i$ is a full satisfaction class and $(\mc{M}, S_i)$ is recursively saturated
	\item $S_{i+1} \restriction F_i = S_i \restriction F_i$,
	\item $S_i$ is $F$-correct on $I$ for sentences from $F_i$, and
	\item for each $x > I$, there is $I < y < x$, $i \in \omega$ and $\phi \in F_i$ such that $F(y, \phi) \in F_i$ and \allowbreak $\lnot (S_i(F(y, \phi), \alpha) \equiv S_i(\phi, \alpha))$ for all assignments $\alpha$.
\end{enumerate}

If $I$ is separable, we also ensure that $S_i$ is $F$-trivial above $I$ for sentences in $F_i$.

Prior to starting the construction, if $I$ has no least $F$-gap above it, we externally fix a sequence $d_0 > d_1 > \ldots$ such that $\inf\{ d_i : i \in \omega \} = I$ and for each $i$, $d_i$ and $d_{i+1}$ are in different $F$-gaps. Additionally, we externally fix an enumeration of $\Form^{\mc{M}}$ (in order type $\omega$).

Suppose we have constructed $F_i$ and $S_i$. Let $\phi$ be the least formula in our enumeration that is not in $F_i$. If $I$ is separable, let $F_{i+1}$ be generated by $F_i$ and $\phi$, and apply Lemma \ref{nonlocal-wsr-lemma} to obtain $S_{i+1}$. Otherwise, we suppose $S_i$ is $F$-correct on $[0, d_i)$ and apply Lemma \ref{nonlocal-fgap-lemma} to obtain $F_{i+1}$, $S_{i+1}$, and $I < c_0 < c_1 < d_i$ such that $S_{i+1}$ is $F$-correct on $[0, c_0)$ but not on $[0, c_1)$. (In fact, there is $\theta \in F_{i+1}$ that witnesses the failure of $F$-correctness on $[0, c_1)$.) Without loss of generality, we can replace $d_{i+1}$ with the minimum of $\{ c_0, d_{i+1} \}$, so that we can assume $S_{i+1}$ is $F$-correct on $[0, d_{i+1})$ and continue.

Having constructed these sequences, let $S = \cup S_i \restriction F_i$. Then it follows that $S$ is $F$-correct on $I$ and for each $x > I$, there is $\phi$ such that $\lnot (T(\phi) \equiv T(F(x, \phi)))$.
\end{proof}

\begin{rem}
    It is easy to see that in fact a tiny modification of our proof of Theorem \ref{nonlocal} shows something more: we can perform our construction in such a way that $S$ is $F$ correct on $I$ not only on all sentences but on \emph{all formulae}. Hence, given $\mc{M}, I$ and $F$ as in the assumptions of Theorem \ref{nonlocal} we can find a satisfaction class $S$ such that 
    \begin{align*}
        I &= \{ x : \forall y \leq x \forall \phi \in \Sent^{\mc{M}} (T(\phi) \equiv T(F(y, \phi))) \}\\
        &= \{ x : \forall y \leq x \forall \phi \in \Form^{\mc{M}} \forall \alpha \bigl(S(\phi,\alpha) \equiv S(F(y, \phi),\alpha)\bigr) \}.
    \end{align*}
\end{rem}
We assume that $\Phi$ has depth 1 in the previous results because the more general case is quite complicated. In particular, if $\Phi$ has depth at least 2, then it might not be possible to ensure that $S$ is $F$-trivial above $I$ as we do in Lemma \ref{nonlocal-wsr-lemma}. For example, suppose $\Phi(p, q) = (\lnot \lnot) q$, $\phi = (0 = 0)$ and $\psi = \lnot (0 = 0)$. Then, for any $x$ and any satisfaction class $S$, $T((\lnot \lnot)^x \phi) \equiv \lnot T( (\lnot \lnot)^x \psi)$. However, we show in our next result that we can still ensure that, if $I$ is separable and closed under addition, there is $S$ such that $I$ is the $(\lnot\lnot)$-correct cut.

\begin{prop}\label{double-neg-nonlocal}
Let $\mc{M} \models \PA$ be countable and recursively saturated and $I \e M$ separable and closed under addition. Then there is $S$ such that $(\mc{M}, S) \models \csm$ and $$I = \{ x : \forall y \leq x \forall \phi \in \Sent^{\mc{M}} (T(\phi) \equiv T((\lnot\lnot)^y (\phi))) \}.$$
\end{prop}

\begin{proof}
Modify the definition of $\triangleleft$ so that $\phi \triangleleft \psi$ if either $\phi$ is an immediate subformula of $\psi$ or $\phi$ does not start with a double negation and there is $x$ such that $(\lnot \lnot)^x \phi = \psi$. (That is, $\phi$ is the $F$ root of $\psi$ where $F(x, \theta) = (\lnot \lnot)^x \theta$.) By similar techinques to the proof of Theorem \ref{nonlocal}, it suffices to show the following: given any finitely generated $X$ and full satisfaction class $S$ such that
\begin{itemize}
\item $(\mc{M}, S)$ is recursively saturated,
\item if $x \in I$, $\phi \in X$, and $(\lnot \lnot)^x \phi \in X$, then for each assignment $\alpha$ of $\phi$, $(\phi, \alpha) \in S$ iff $((\lnot \lnot)^x \phi, \alpha) \in S$, and,
\item if $x > I$, $(\lnot\lnot)^x \phi \in X$, and $\phi \in X$, then for each assignment $\alpha$ of $\phi$, $(\lnot \phi, \alpha) \in S$ iff  $((\lnot \lnot)^x \phi, \alpha) \in S$, 
\end{itemize} then for any finitely generated $X^\prime \supseteq X$, there is a full satisfaction class $S^\prime$ such that 
\begin{itemize}
\item $(\mc{M}, S^\prime)$ is recursively saturated,
\item $S^\prime \restriction X = S \restriction X$,
\item if $x \in I$, $\phi \in X^\prime$, and $(\lnot \lnot)^x \phi \in X^\prime$, then for each assignment $\alpha$ of $\phi$, $(\phi, \alpha) \in S^\prime$ iff $((\lnot \lnot)^x \phi, \alpha) \in S^\prime$, and,
\item if $x > I$, $(\lnot\lnot)^x \phi \in X^\prime$, and $\phi \in X^\prime$, then for each assignment $\alpha$ of $\phi$, $(\lnot \phi, \alpha) \in S^\prime$ iff $((\lnot \lnot)^x \phi, \alpha) \in S^\prime$.
\end{itemize} Moreover, rather than find a full satisfaction class satisfying the above, we simply need to find an $X^\prime$-satisfaction class $S^\prime$ satisfying the above. To do so, let $a, b, $ and $c$ code enumerations such that $\{ (c)_n : n \in \omega \} = X^\prime \cap \Sent^{\mc{M}}$, $(b)_n$ is the root of $(c)_n$, and $(c)_n = (\lnot \lnot)^{(a)_n} ((b)_n)$. By separability, there is $d$ such that for each $n \in \omega$, $(a)_n \in I$ if and only if $(a)_n < d$. We show that the theory $Th$ consisting of the following is consistent:
\begin{itemize}
\item $\textnormal{ElDiag}(\mc{M})$,
\item $\{ \Comp(\phi, \psi, \theta) : \phi, \psi, \theta \in X^\prime \}$,
\item $\{ S^\prime(\phi, \alpha) \equiv S(\phi, \alpha) : \phi \in X \}$ (preservation),
\item $\{ S^\prime( (\lnot \lnot)^{(a)_n} ((b)_n), \alpha) \equiv S^\prime((b)_n, \alpha) : n \in \omega, (a)_n < d \}$ ($F$-correctness), and
\item $\{ S^\prime( (\lnot \lnot)^{(a)_n} ((b)_n), \alpha) \equiv S^\prime(\lnot (b)_n, \alpha) : n \in \omega, (a)_n > d \}$ ($F$-incorrectness).
\end{itemize} Again, one can show that if $(\mc{M}, S, S^\prime) \models Th$, then $S^\prime$ is an $X^\prime$-satisfaction class satisfying the required properties.

To show that $Th$ is consistent, let $T_0 \subseteq Th$ be finite, and let $C$ be the set of formulas whose instances of compositionality, preservation, double negation correctness and double negation incorrectness are in $T_0$. Since $C$ is finite, then the modified subformula relation $\triangleleft$ is well-founded on $C$, and we define $S^\prime$ inductively on this relation.

Suppose $\phi$ is minimal in $C$. If $\alpha$ is an assignment for $\phi$, we put $(\phi, \alpha) \in S^\prime$ if either $\phi$ is atomic and $\mc{M} \models \phi[\alpha]$, or $\phi \in X$ and $(\phi, \alpha) \in S$. We define $\phi$ of higher rank using compositionality if possible. If this is not possible, then it must be the case that there is $n \in \omega$ such that $\phi = (\lnot \lnot)^{(a)_n} ((b)_n)$ and $(b)_n \in C$ has lower rank than $\phi$. We put $(\phi, \alpha) \in S^\prime$ if either $(a)_n < d$ and at an earlier stage we decided $((b)_n, \alpha) \in S^\prime$, or if $(a)_n > d$ and, at an earlier stage we decided $((b)_n, \alpha) \not \in S^\prime$.

We verify that $(\mc{M}, S, S^\prime) \models T_0$. Clearly it satisfies the diagram and compositionality axioms by construction. Suppose $\phi \in X$ is such that $\forall \alpha (S^\prime(\phi, \alpha) \equiv S(\phi, \alpha)) \in T_0$. If $\phi$ is of minimal rank, then this is true by construction. If not, we can assume, by induction, that whenever $\psi \triangleleft \phi$ is such that $\psi \in C$, then $\forall \alpha (S^\prime(\psi, \alpha) \equiv S(\psi, \alpha))$. If $\phi$ is determined via compositionality, then the result for $\phi$ follows from the fact that both $S$ and $S^\prime$ are compositional for formulas in $X$. Otherwise, the result for $\phi$ follows from either double negation correctness up to $I$, or double negation incorrectness above $I$.

Now let $\theta = (\lnot \lnot)^{(a)_n}((b)_n)$, and suppose $\forall \alpha S^\prime(\theta, \alpha) \equiv S^\prime((b)_n, \alpha) \in T_0$, where $(a)_n < d$. By construction, $\theta$  is not minimal in $C$. The immediate subformula of $\theta$ is $\psi = \lnot (\lnot \lnot)^{(a)_n - 1} ((b)_n)$. If $\psi \in C$, then by construction we have that $S^\prime(\theta, \alpha) \equiv \lnot S^\prime(\psi, \alpha)$. By induction, we can assume we have $S^\prime(\psi, \alpha) \equiv \lnot S^\prime((b)_n, \alpha)$. If $\psi \not \in C$, then by construction we put $S^\prime(\theta, \alpha) \equiv S^\prime((b)_n, \alpha)$.

A similar argument shows double negation incorrectness in the case that $(a)_n > d$.
\end{proof}


By Theorem \ref{nonlocal}, if $I$ is either separable or has no least $\mathbb{Z}$-gap above it, there is $T$ such that $(\mc{M}, S) \models \csm$ and $I(\IDC_S) = I$. In fact, if $\omega$ is a strong cut, then by Proposition \ref{wsr-vs-least-z} every cut $I$ is either separable or has no least $\mathbb{Z}$-gap, and therefore every cut $I$ can be $I(\IDC_S)$ for some satisfaction class $S$.  Similarly, if $\omega$ is strong, then every additively closed cut $I$ is either separable or has no least additive gap above it, and therefore each additively closed cut can be $I(\IDC^{\text{bin}}_S)$.

To complete the picture, we can show that if $F$ is an idempotent sentential operator and $I$ is the $F$-correct cut, then either $I$ has no least $F$-gap above it or is separable. Therefore, if $\mc{M}$ is not arithmetically saturated, then there are cuts $I$ which cannot be realized as $I(\IDC_S)$ for any $T$.

\begin{prop}\label{wsr-or-no-least-z}
Let $F$ be an accessible idempotent sentential operator. Suppose $(\mc{M}, S)\models\csm$ and $I\subseteq_{end} \mc{M}$ is such that
\[I = \{x : \forall y\leq x \forall \phi \bigl(T(F(y,\phi))\equiv T(\phi)\bigr)\}.\]
Then either there is no least $\mathbb{Z}$-gap above $I$ or $I$ is separable.
\end{prop}
\begin{proof}
Assume that there is a least $\mathbb{Z}$-gap above $I$ and fix $a$ coding a sequence such that $(a)_{n+1} = (a)_n -1$ and $\inf_{n \in \omega} \{ (a)_n \} = I$. Since $(a)_0\notin I$ there is $\phi$ such that $(\mc{M}, S)\models \neg T(F((a)_0,\phi)\equiv \phi)$. By the properties of $F$ it follows that for every $n \in \omega$, $(\mc{M}, S)\models \neg T(F((a)_n,\phi)\equiv \phi)$. Let $D = \{F(a,\phi)\equiv \phi: a < (a)_0 \}$ and let $A = \{F(a,\phi)\equiv \phi: a\in I\}$. It follows that for every $c < (a)_0$, $(\mc{M},S)\models T(F(c,\phi)\equiv \phi)$ iff $F(c,\phi)\equiv\phi \in A.$ So by Theorem \ref{separability-theorem}, $A$ is separable from $D$; therefore $I$ is separable.
\end{proof}

This completes the picture for accessible $F$. In particular, we have a complete picture for which cuts can be $I(\IDC_S)$. If $\omega$ is strong, then every cut can be $I(\IDC_S)$ for some $S$, and if $\omega$ is not strong, then only those cuts which have no least $\mathbb{Z}$-gap above it can be $I(\IDC_S)$. What about for cuts which are $F$-correct for additive $F$, like $I(\QC_S)$?

\begin{lem}
Suppose $(\mc{M}, S)\models \csm$, $t\in M$ is a full binary tree of height $c$ labelled with sentences, such that $t\upharpoonright_T:= \{s\in \{0,1\}^{<\omega}\ \ | \ \ T(t(s))\}$ has arbitrarily long branches. Then $t\upharpoonright_T$ has an infinite coded branch.
\end{lem}
\begin{proof}
Consider the following sequence of formulae
\[\phi_n(x):= \bigwedge_{s:\len(s)\leq n} \bigl(t(s)\equiv s\in x\bigr).\]
The above conjunction is of the form $(\phi_{s_0} \wedge (\phi_{s_1} \wedge (\ldots )\ldots)$ where $\{s_i\}_{i<{2^n}}$ is an enumeration of all binary sequences of length $\leq n$ according to the length-first lexicographic ordering. By Smith's result \cite[Theorem 2.19]{smith-nonstandard-def} there is  $a\in M$ such that for all $n \in \omega$, $T(\phi_n(a))$ holds. Hence $\{s\in \{0,1\}^{<\omega}\ \ | \ \ s\in a\}$ is an infnite finitely branching tree, so it has a coded infinite branch, $b$. Since $b\subseteq a$, for every $i\in \omega$ we have $(\mc{M}, S)\models T(b(i))$.
\end{proof}

\begin{prop}\label{wsr-or-no-least-plus}
Let $F$ be an additive idempotent sentential operator. Suppose $(\mc{M}, S)\models\csm$ and $I\subseteq_{end} \mc{M}$ is such that
\[I = \{x : \forall y\leq x \forall \phi \bigl(T(F(y,\phi))\equiv T(\phi)\bigr)\}.\]
Then either there is no least $+$-closed gap above $I$ or $I$ is separable.
\end{prop}
\begin{proof}
Suppose there is a least $+$-closed gap above $I$ and let $a$ code a sequence such that $(a)_{n+1} = \lfloor\frac{(a)_n}{2}\rfloor$ and $\inf_{n\in\omega} (a)_n = I.$ Let $c$ be the length of $a.$ Observe that $\sup(I\cap im(a)) = I$, so by Proposition \ref{prop_supremum_separabilit} it is sufficient to show that $I\cap im(a)$ is separable. Fix $\phi$ such that $(\mc{M}, S)\models \neg T(F((a)_0,\phi)\equiv \phi)$. Then for every $n$ it holds that
\[(\mc{M}, S)\models \neg T(F((a)_{n+1},\phi)\equiv \phi)\vee  \neg T\bigl(F((a)_{n+1},F((a)_{n+1},\phi))\equiv F((a)_{n+1},\phi)\bigr).\]
 Define the labelling $t$ of a full binary tree of height $c$ by recursion as follows:

\begin{align*}
     t_\varepsilon =& \neg (F((a)_0,\phi)\equiv \phi) &\\
     t_{s^\frown 0} =& \neg \bigl(F((a)_{n+1},t_s^{*})\equiv t_s^{*}\bigr)& \text{if }\len(s) = n\\
     t_{s^\frown 1} =& \neg \bigl(F((a)_{n+1},F((a)_{n+1},t_{s}^*))\equiv F((a)_{n+1},t_{s}^*)\bigr) & \text{if }\len(s) = n
 \end{align*}
 In the above, $x^*$ is the unique sentence $\psi$ such that there is $\theta$ such that $x = \neg (\theta\equiv \psi)$. By our assumption, $t\upharpoonright_T$ has arbitrarily long branches, so there is an infinite coded branch $b$ of $t$ such that for every $i\in\omega$ $(\mc{M}, S)\models T(b(i))$. Moreover, by the construction of $t$, for every $i\in \text{dom}(b)$, 
 \begin{center}
    $(\mc{M}, S) \models T(b(i))$ iff $i\in\omega.$
 \end{center}
 It follows that the set $A=\{\psi \in im(b) : T(\neg\psi)\}$ is separable.
 Observe that for every $i < \len(b)$ we have 
$$(a)_i\in I \iff T(\neg b(i)) \iff b(i)\in A.$$ 
 Hence $im(a)\cap I = G^{-1}[A]$, where $G$ is the definable function $(a)_i \mapsto b(i)$. By Proposition \ref{prop_closure_sep_preim} this ends the proof.
\end{proof}

\begin{cor}\label{arithsat-accessible}
For a countable, recursively saturated $\mc{M} \models \PA$, the following are equivalent:
\begin{enumerate}
    \item\label{arithsat} $\mc{M}$ is arithmetically saturated, and
    \item\label{every-cut-idc} For every idempotent sentential operator $F$ with template $\Phi(p, q)$ of depth 1, and every $F$-closed cut $I \e M$, there is $S$ such that $(\mc{M}, S) \models \csm$ and $$I = \{ x : \forall y \leq x \forall \phi \bigl(T(F(y, \phi)) \equiv T(\phi)\bigr) \}.$$
\end{enumerate}
\end{cor}

Note that the implication $\eqref{every-cut-idc} \implies \eqref{arithsat}$ holds in more generality: it does not rely on $\Phi$ having syntactic depth $1$.

\begin{proof}
We show $\eqref{arithsat} \implies \eqref{every-cut-idc}$. Suppose $\omega$ is a strong cut. Let $a \odot n$ be $a - n$, if $F$ is accessible, and $\lfloor \frac{a}{n} \rfloor$, if $F$ is additive. By Proposition \ref{wsr-vs-least-z}, if $I$ is not separable, then $I$ is not $\omega$-coded, and so there is no $a > I$ such that $\inf(\{a \odot n : n \in \omega \}) = I$. Therefore, every $F$-closed cut $I$ is either separable or has no least $F$-gap above it. The result follows from Theorem \ref{nonlocal}.

Conversely, if $\mc{M}$ is not arithmetically saturated, let $I \e M$ be any cut with a least $F$-gap above it. For example, fix a nonstandard $c$ and let $I = \inf( \{ c \odot n : n \in \omega \})$. Since $\omega$ is not strong, by Proposition \ref{wsr-vs-least-z}, $I$ is not separable. It follows by Proposition \ref{wsr-or-no-least-z} for accessible $F$, and by Proposition \ref{wsr-or-no-least-plus} for additive $F$, that there is no $S$ such that $(\mc{M}, S) \models \csm$ and $$I = \{ x : \forall y \leq x \forall \phi (T(F(y, \phi)) \equiv T(\phi)) \}.$$
\end{proof}

\section{Disjunctively Correct Cut}\label{dc-section}


We proceed to the strongest correctness property, that of full disjunctive correctness $(\dc_S)$. As usual we shall focus on $I(\dc_S)$. The first proposition states that the intuitive strength of full disjunctive correctness is reflected in the closure properties of $I(\dc_S)$:

\begin{prop}
For every $(\mc{M},S)$, $I(\dc_S)$ is closed under multiplication.
\end{prop}
\begin{proof}
We shall use a result from \cite{cieslinski-lelyk-wcislo-dc}: define the sequential induction cut $\SInd_S$ to be the set of those $c$ such that the following is true in $(\mc{M},S):$
\[\forall x\leq c \forall \anglebracket{\phi_i: i\leq x}\bigl(T(\phi_0)\wedge \forall y< x (T(\phi_y)\rightarrow T(\phi_{y+1}))\bigr)\rightarrow \forall i\leq x T(\phi_i).\]
Then the proof of \cite[Theorem 8]{cieslinski-lelyk-wcislo-dc} directly shows that $\dc_S\subseteq \SInd_S$. Now we proceed to the main argument: fix any $c\in \dc_S$ and let $b\leq c^2$. Fix any $d,r$ such that $b = dc+r$ and $r < c$. Fix any $\anglebracket{\phi_i: i\leq b}$ and assume first that $T(\bigvee_{i\leq b}\phi_i)$ and, aiming at a contradiction that for every $i\leq b$, $T(\neg \phi_i)$. Define the auxiliary sequence of length $d$: for each $i\leq d$ let $\theta_i = \bigvee_{j\leq ic}\phi_j$ and let $\theta_{d+1} = \phi_b.$ We show that for every $i< d+1$, $T(\neg \theta_i)\rightarrow T(\neg \theta_{i+1}).$ Fix any $i$ and assume $T(\neg \theta_i)$. Let $c'$ be $c$ if $i<d$ and $r$ if $i = d$. Consider the sequence $\psi_k = \bigvee_{j\leq ic+k}\phi_j$. We claim that for any $k< c'$ $T(\neg\psi_k)\rightarrow T(\neg\psi_{k+1}).$ Indeed, fix any $k<c'$ and assume $T(\neg\psi_k)$. Observe that by the definition of $T$, the definition of $\psi_{k+1}$ and the compositional axioms we have 
\begin{multline*}
T(\neg\psi_{k+1})\equiv S(\neg\psi_{k+1},\emptyset) \equiv S(\neg (\psi_k\vee \phi_{ic+k+1}),\emptyset)\equiv S(\neg \psi_k,\emptyset)\wedge S(\neg\phi_{ic+k+1},\emptyset) \\
\equiv T(\neg\psi_k)\wedge T(\neg\phi_{ic+k+1}).
\end{multline*}
The last sentence is clearly true by our assumptions. Hence, since $c' \in \SInd_S$, we conclude that $T(\neg \psi_{c'})$. Since by definition $\psi_{c'} = \theta_{i+1}$, we established that for any $i<d$ $T(\neg\theta_{i})\rightarrow T(\neg\theta_{i+1}).$ Since $d+1\in \SInd_S$, we conclude that $T(\neg\theta_{d+1})$. By definition, we obtain that $T(\neg \bigvee_{i\leq b} \phi_i)$, which contradicts our assumption. 

Now assume that for some $e\leq b$, $T(\phi_e)$ holds. In particular, it holds that $T(\bigvee_{i\leq e}\phi_i)$. Let us fix $d', r'$ such that $b-e = d'c + r'$ and for $j\leq d'$ define $\theta_{j} = \bigvee_{i\leq e+ jc}\phi_i$ and $\theta_{d'+1} = \bigvee_{i\leq b}\phi_i$. As in the above proof we show that for each $j\leq d'$, $T(\theta_j)\rightarrow T(\theta_{j+1})$ and obtain $T(\bigvee_{i\leq b}\phi_j)$, which concludes the proof of the reverse implication and the whole argument.
\end{proof}

We conclude with a limitative result which shows that methods used to prove the main results of previous sections are unsufficient for obtaining the analogous results in the context of $\dc_S$. This is because, as conjectured, our methods show that, in an arithmetically saturated model, any cut can be characterized as $I(\IDC_S)$ for some regular satisfaction class which satisfies the internal induction axiom, For such a satisfaction class $S$, $S(\phi,\emptyset)$ behaves like a truth predicate satisfying the axioms of $\textnormal{CT}^-$ and we have the following small insight. Below $\textnormal{Con}_{\PA}(x)$ is a formula with a free variable $x$ which canonically expresses that there is no proof of $0=1$ in $\PA$ whose code is smaller than $x$.

\begin{prop}
Suppose that $(\mc{M},S)\models \csm$, $S$ is regular and satisfies the internal induction axiom. Then, for every $a\in \dc_S$, $\mc{M}\models \textnormal{Con}_{\PA}(a).$
\end{prop}
\begin{proof}[Sketch]
Let $\textnormal{CT}^-(x)$ denote a formula of the language $\mc{L}_{\PA}\cup\{T\}$ with a free variable $x$ which expresses ''$T(x)$ satisfies Tarski's inductive truth conditions for sentences of logical depth at most $x$''. By inspection of the proof of Theorem 3.1 from \cite{wcislyk_notes_on_bounded} one sees that if $a\in \dc_S$, then there is (typically nonstandard) $\psi \in \mc{M}$ with a unique free variable such that $\mc{M}\models \Form_{\mc{L}_{\PA}}(\psi(x))$ and $(\mc{M}, S)\models \textnormal{CT}^-(a)[T*\psi(x)/T(x)].$ $T*\psi(x)$ denotes a formula with a free variable $x$ which expresses ''The result of substituting the numeral of $x$ for the unique free variable in $\psi$ is true'' (we use the notation from \cite{wcislyk_mpt}, Lemma 3.6) and $\textnormal{CT}^-(a)[T*\psi(x)/T(x)]$ is the formula obtained by substituting $T*\psi(x)$ for $T(x)$. As in \cite{wcislyk_mpt}, Lemma 3.7 we conclude that $T*\psi(x)$ satisfies full induction scheme in $(\mc{M}, S)$. It follows that no proof with a code less than $a$ can be the proof of $0=1$ from the axioms of $\PA$, because each formula in this proof is of complexity at most $a$ and all the premises are made true by $T*\psi(x)$.
\end{proof}

\section{Appendix}\label{app-section}
In this Appendix we indicate how to modify the proof of Theorem \ref{truth-closed-theorem} in order to obtain a much better-behaved satisfaction class. In particular we would like the constructed satisfaction classes to define a truth predicate. We start with introducing the notion of \textit{regularity}. The definition is taken from \cite{wcislo_satisfaction_definability}:

\begin{defn}
For every formula $\phi$ and a term substitution $\gamma$, $\phi[\gamma]$ denotes the result of subtituting $\gamma(v)$ for every free occurrence of $v$ in $\phi$, for every $v$ in the domain of $\gamma.$

We shall treat assignments as substitution of numerals: if $\alpha$ is an assignment, then by writing $\phi[\alpha]$ we treat $\alpha$ as a substitution which to every $v$ assigns the canonical numeral naming $\alpha(v)$ (i.e. the term expressing the sum of $0$ and $\alpha(v)$-many 1's). 
\end{defn}

For example, if $\alpha(v_0) = 3$ and $\alpha(v_1) = 1$, then $(\exists v_0(v_0 = v_1)\vee v_0+1 = v_2)[\alpha] = \exists v_0(v_0 = 0+1)\vee 0+1+1+1+1 = v_2.$

\begin{defn}[$\PA$] \label{def_structural_template}
	 If $\phi \in \Form_{\mc{L}_{\PA}}$, we say that $\widehat{\phi}$ is its \emph{structural template} iff 
	\begin{itemize}
	  	\item No constant symbol occurs in $\widehat{\phi}$.
		\item No free variable occurs in $\widehat{\phi}$ twice.
		\item No complex term containing only free variables occurs in $\widehat{\phi}$. 
		\item No variable occurs in $\widehat{\phi}$ both as a bound and as a free variable. 
		\item The formula $\phi$ can be obtained from $\widehat{\phi}$ by renaming bound variables and substituting terms for free variables in such a way that no variable appearing in those terms becomes bound.
		\item $\widehat{\phi}$ is the smallest formula with those properties (recall that we identify formulae with their G\"odel codes).
	\end{itemize}
We say that formulae $\phi, \psi$ are \emph{structurally similar}, $\phi \sim \psi$ iff $\widehat{\phi} = \widehat{\psi}.$ 


Suppose that $\kappa$ is an \textit{occurrence of a subformula} of $\phi$ (not necessarily direct). With $\kappa_{\widehat{\phi}}$ we denote the subformula of $\widehat{\phi}$ whose occurrence in $\widehat{\phi}$ corresponds to $\kappa$ (recall that $\phi$ and $\widehat{\phi}$ have the same syntactic structure). For a formula $\psi$, $[\psi]_{\widehat{\phi}}$ denotes the set $\{\kappa_{\widehat{\phi}} \ \ : \ \ \kappa \textnormal{ is an occurrence of $\psi$ in $\phi$}\}$.
\end{defn}

Note that the definition of structural similarity formalizes in $\PA$ and the relation is an equivalence relation, provably in $\PA$. Moreover we can assume that if $\phi$ is of standard complexity, then $\widehat{\phi}$ is a standard formula.

\begin{ex}
The structural template of $0=0\vee 0=0$ is $v_0=v_1\vee v_2 = v_3$, while the syntactic template of $\exists v_2(v_2+1 = v_1+1+1)$ is $\exists v_0(v_0+v_1 = v_2)$, where in both cases $v_i$ are chosen in such a way to minimize the formula. $0=0_{\widehat{0=0\vee 0=0}} = \{v_0=v_1, v_2=v_3\}.$

Formulae $\forall v_0(v_0 = v_1+1) \vee \neg(v_1 = v_0+1)$ and $\forall v_3(v_3 = v_2+1+1)\vee \neg (v_2+1 = v_0)$ are structurally similar.
\end{ex}

\begin{rem}[$\PA$]
For every two formulae $\psi,\phi$ such that $\psi$ is a subformula of $\phi$ (not necessarily direct), $\widehat{\psi}$ differs from every formula from the set $[\psi]_{\widehat{\phi}}$ at most by a permutation of free variables and renaming bound variables. For every $\theta\in [\psi]_{\widehat{\phi}}$ we shall denote with $\sigma_{\theta,\widehat{\psi}}$ the permutation of free variables such that $\sigma_{\theta,\widehat{\psi}}[\theta] = \widehat{\psi}$. 
\end{rem}

\begin{defn}[$\PA$]
Let $\phi$ be any formula and $\gamma$ be a term substitution such that $\widehat{\phi}[\gamma]$ differs from $\phi$ only modulo renaming the bound variables. Then for every assignment $\alpha$ for $\phi$ let $\widehat{\alpha}_{\phi}$ be the assignment for $\widehat{\phi}$ given by $\widehat{\alpha_{\phi}}(v) = \gamma(v)^{\alpha}$. We recall that for a term $t$ and assignment $\alpha$, $t^{\alpha}$ denotes the value of term $t$ under assignment $\alpha$.
\end{defn}

In other words, $\widehat{\alpha}_{\phi}$ assigns to a variable $v$, the value of the term $\gamma(v)$ under the assignment $\alpha$. For illustration assume that $\theta$ is either a true atomic sentence or the negation of a true atomic sentence and $F$ is a local idempotent operator for $\theta$ with a template $\Phi(p,q)$ (as in Definition \ref{local-id}). Then for any $x$, $\widehat{F(x)}$ can differ from $F(x)$ only in that
\begin{itemize}
    \item $\widehat{F(x)}$ may use different free and bound variables;
    \item each element of $[\theta]_{\widehat{F(x)}}$ is of the form $v_i = v_j$ for some variables $v_i$ and $v_j$ (if $\theta$ is a true atomic sentence) or each element of 
    $[\theta]_{\widehat{F(x)}}$ is of the form $\neg v_i = v_j$ for some variables $v_i$ and $v_j$ (if $\theta$ is the negation of a true atomic sentence. Moreover all the variables in $\widehat{F(x)}$ occur only in formulae from $[\theta]_{\widehat{F(x)}}$. In particular $\widehat{F(x)}$ is not a sentence.
\end{itemize}
Moreover, observe that, since $F(x)$ is a sentence, then $\emptyset$ is the unique assignment for $F(x)$. Hence, if $\theta$ is either $s = t$ or $\neg s=t$, where $s$ and $t$ are closed terms whose value is $a$, then $\widehat{\emptyset_{F(x)}}$ is constantly equal to $a$.

The above described situation of a local idempotent operator for $\theta$ will be the only one which we shall consider in this section.

\begin{defn}
    An $X$-satisfaction class $S$ is \emph{regular} if for every formulae $\phi, \widehat{\phi}\in X$ and every assignment $\alpha$ for $\phi$, $(\phi,\alpha)\in S$ iff $(\widehat{\phi},\widehat{\alpha}_{\phi})\in S$.
\end{defn}


We are now proceeding to strengthening Theorem \ref{truth-closed-theorem}. For notational reasons we write $\mc{M},\alpha \models \phi$ instead of $\mc{M}\models \phi[\alpha]$ to mean that a formula $\phi$ is satisfied in $\mc{M}$ by an assignment $\alpha$.

\begin{defn}
  Fix $\mc{M}\models \PA$, $X\subseteq M$, $\theta$, $F$ and $\Phi$ such that $F$ is local idempotent sentential operator for $\theta$ with syntactic template $\Phi(p,q)$. 
  \begin{enumerate}
      \item We say that a formula $\phi$ is an $F$-\textit{intermediate formula} if for some $x,$ $F(x)$ is a subformula of $\phi$ (not necessarily direct or proper) and $\phi$ is a subformula (not necessarily direct or proper) of $F(x+1)$. 
      \item For an intermediate formula $\phi$, the $F$-length of $\phi$ is the maximal $x$ such that $F(x)$ is a subformula of $\phi$. 
      \item Recall that $\textnormal{compl}(\phi)$ denotes the complexity of a formula $\phi$ (defined in Preliminaries). For an $F$-intermediate formula $\phi$, assignment $\alpha$ for $\widehat{\phi}$ and $x$ such that for some $n\in\omega$, $\textnormal{compl}(\phi) = \textnormal{compl}(F(x)) + n$, we say that $\alpha$ $(X,x)$-\textit{satisfies} $\widehat{\phi}$ if $\mc{M},\alpha\models\widehat{\phi}[A/F(x)]$ where $A$ is $0=0$ if $x\in X$ and $0=1$ otherwise and $\widehat{\phi}[A/F(x)]$ denotes the result of replacing in $\widehat{\phi}$ every occurrence of $F(x)_{\widehat{\phi}}$ with $A$. We say that $\alpha$, $X$-\textit{satisfies} $\widehat{\phi}$ if $\alpha$ $(X,x)-$satisfies $\phi$ where $x$ is the $F$-length of $\phi$.
  \end{enumerate}
\end{defn}

We note that the above definition makes sense, since $\widehat{\phi}[A/F(x)]$ is a formula of standard complexity (possibly with variables with nonstandard indices).

\begin{prop}\label{stupid_proposition}
Fix any $\mc{M}\models \PA$ and $X\subseteq M$ which is closed predecessor. For an arbitrary intermediate formula $\phi$ of nonstandard complexity and assignment $\alpha$ for $\widehat{\phi}$ the following are equivalent:
\begin{enumerate}
    \item $\alpha$ $X$-satisfies $\widehat{\phi}$.
    \item For every $x$ such that $\textnormal{compl}(\phi)-\textnormal{compl}(F(x))\in\omega$, $\alpha$ $(X,x)$-satsfies $\widehat{\phi}$.
    \item For some $x$ such that $\textnormal{compl}(\phi)-\textnormal{compl}(F(x))\in\omega$, $\alpha$ $(X,x)$-satsfies $\widehat{\phi}$.
\end{enumerate}
\end{prop}
\begin{proof}
    Follows immediately from the definition of $F$ and the fact that $\theta$, $\Phi$ are chosen so that $\Phi(\theta,q)$ is equivalent to $q$.
\end{proof}

\begin{thm}\label{truth-closed-reg-theorem}
Let $\theta$ be either a true atomic sentence or a negation of a true atomic sentence and $F$ a local idempotent sentential operator for $\theta$ with template $\Phi(p,q)$. Let $X \subseteq M$ be separable, closed under successors and predecessors, and for each $n \in \omega$, $n \in X$ if and only if $\mc{M} \models \theta$. Then $\mc{M}$ has an expansion $(\mc{M}, S) \models \csm$ such that $X = \{ x \in M : (\mc{M}, S) \models T(F(x)) \equiv T(\theta) \}$ and $S$ is a regular satisfaction class.
\end{thm}

\begin{proof}
The initial structure of the argument is very similar to that used in proving Theorem \ref{truth-closed-theorem}. Let $D = \{ F(x) : x \in M \}$ and $A = \{ F(x) : x \in X \}$. Note that $A$ is separable from $D$. We build sequences $F_0 \subseteq F_1 \subseteq \ldots$ and $S_0, S_1, \ldots$ such that:
\begin{itemize}
	\item each $F_i$ is a finitely generated set of formulas such that $\cup F_i = \Form^{\mc{M}}$,
	\item each $S_i$ is a regular full satisfaction class and $(\mc{M}, S_i)$ is recursively saturated,
	\item $S_{i+1} \restriction F_i = S_i \restriction F_i$, and
	\item for each $\phi \in D \cap F_i$, $(\phi, \emptyset) \in S_i$ if and only if $\phi \in A$.
\end{itemize}
Given such a sequence, $S = \cup (S_i \cap F_i\times M)$ would be the required full satisfaction class on $\mc{M}$.

Externally fix an enumeration of $\Form^{\mc{M}}$ in order type $\omega$. We can assume, without loss of generality, that $\theta$ appears first in this enumeration. We let $F_0$ be $\{\theta\}$ and $S_0$ be any regular full satisfaction class which satisfies internal induction. Let $F_{i+1}$ be generated by $F_i$ and the least $x \in \Form^{\mc{M}} \setminus F_i$ in the aforementioned enumeration. Let $F^\prime = F_i \cup (F_{i+1} \cap D)$. Let $a$ be a sequence such that $\{ F((a)_n) : n \in \omega \} = F_{i+1} \cap D$. Note that such a sequence exists since $F_{i+1}$ is finitely generated. Let $c$ be as in the definition of separability for $a$.

We shall now show how to construct an elementary extension $\mathcal{N}$ of $\mathcal{M}$ and a full regular satisfaction class $S'$ on $\mathcal{N}$ such that
\begin{itemize}
    \item $S'\cap F_i\times M=S_i\cap (F_i\times M)$.
    \item For each $n\in\omega$, $(F((a)_n), \emptyset)\in S'\iff \mc{M}\models n\in c$.
\end{itemize}
By a straightforward resplendence argument one can then copy $S'$ to $\mc{M}$. This last step crucially uses the fact that $(\mc{M}, S_i)$ is recursively saturated and the facts that (1) $F_{i+1}$ is finitely generated and (2) that we can code the membership in $F_{i+1}\cap X$ via the parameter $c$. The construction of $S'$ follows the lines of a standard Enayat--Visser construction (as presented in \cite{enayat-visser}): we build a sequence of models $\mc{M} = \mc{M}_0\preceq\mc{M}_1\preceq \mc{M}_2,\ldots$ and sets $S'_1, S'_2,\ldots$ such that
\begin{enumerate}
    \item $S'_i\subseteq M_i^2$, $S_i\cap (F_i\times M) = S'_1\cap (F_i\times M)$ and for all $i>0$, $S'_{i+1}\cap M_{i-1}^2 = S'_i\cap M_{i-1}$
    \item $(\mc{M}_{i+1}, S'_{i+1})\models\csm\restr{\Form^{\mc{M}_i}}$;
    \item for each $n\in\omega$, $(F((a)_n), \emptyset)\in S'_i\iff \mc{M}\models n\in c$
    \item for every $\phi\in \Form^{\mc{M}_i}$ and $\alpha \in M_{i+1}$, if $\alpha$ is an assignment for $\phi$, then
    \[(\phi,\alpha)\in S'_{i+1}\iff (\widehat{\phi},\widehat{\alpha}_{\phi})\in S'_{i+1}.\]
\end{enumerate}
Then one easily checks that for $\mc{N} = \bigcup_{i}\mc{M}_i$ and $S' = \bigcup_i S'_{i+1}\cap (\Form^{\mc{M}_i}\times M_{i+1})$, $(\mc{N},S')$ satisfy the conditions A,B,C above and  $S'$ is a full regular satisfaction class. We note that condition $4$ does not contradict the fact that $S'_{i+1}$ is defined only for formulae in $M_{i}$, because the operations $\widehat{\cdot}$ are $\mc{L}_{\PA}$ definable, so if $\phi$ in $M_{i}$ then $\widehat{\phi}\in M_i.$ 

We show how to construct $\mc{M}_1$ and $S'_1$ and the rest of cases is fully analogous (but simpler because we do not have to care about condition (3) from the above list). Consider the theory in the language $\mc{L}_{\mc{M}}\cup \{S_1'\}$ which is given as the union of the following sets:
\begin{itemize}
    \item[i] ElDiag($\mc{M}$)
    \item[ii] $\{\textnormal{Comp}(\phi,\psi,\theta)\ \ : \phi,\psi,\theta\in \Form^{\mc{M}}\}$ 
    \item[iii] $\{\forall \alpha \bigl(S_1'(\phi,\alpha)\equiv S_1'(\psi,\widehat{\alpha}_{\phi})\bigr)\ \ : \ \  \phi,\psi\in\Form^{\mc{M}}, \psi = \widehat{\phi}\}$.
    \item[iv] $\{S'_1(F((a)_n),\emptyset)\equiv n\in c\ \ : n\in\omega\}$.
    \item[v] $\{S'_1(\phi,\alpha) \ \ : \phi\in F_i, (\phi,\alpha)\in S_i\}$
\end{itemize}
We argue that the above theory is consistent, which is enough to obtain $\mc{M}_1$ and $S_1$. So fix $A$ - a finite portion of the above theory. Let $B$ consists of all $\phi\in \Form^{\mc{M}}$ which occur in one of the axioms in $A.$ 
We build the extension of $S_1'\subset M^2$ such that $(\mc{M}, S_1')\models A$ by induction on the complexity of $\phi\in B$. We note that this is meaningful, since $B$ is finite. Moreover we always define $S_1'$ on $\widehat{\phi}$ and then extend $S_1'$ canonically to all formulae in $\sim$ equivalence class. In the construction we shall not refer to the fragment of $X$ given by $c$ and $a$, but rather to the whole of $X$. $c$ and $a$ were introduced to enable the resplendency argument.

Assume $\phi$ has the least possible complexity among formulae in $B$. We put $(\widehat{\phi},\alpha)\in S_1'$ iff  $\alpha$ is an assignment for $\widehat{\phi}$ and one of the following holds:
\begin{itemize}
    \item[a.] $\widehat{\phi}$ is standard and $\mathcal{M},\alpha\models \widehat{\phi}$.
    \item[b.] $(\widehat{\phi},\alpha)\in S_i$ and $\phi\in F_i.$
    \item[c.] $\alpha$ is a constant function, $\phi$ is an $F$-intermediate formula and $\alpha$ $X$-satisfies $\widehat{\phi}$. 
\end{itemize}
Then, for every formula $\psi\in B$ which has the least possible complexity, we put $(\psi,\alpha)\in S_1'$ iff $(\widehat{\psi},\widehat{\alpha}_{\psi})\in S_1'.$ The base step of our induction process is finished.

Now for $\psi\in B$ we assume that for every $\phi\in B$ of complexity lower than the complexity of $\psi$ and every $\psi'$ such that $\psi'\sim \psi$, $S_1'$ has been defined. If all immediate subformulae of $\psi$ are in $B$, then by induction we can assume that $S_1'$ is defined for their templates and so we can extend $S_1'$ to $\widehat{\psi}$ using the compositional clauses. Otherwise,  we put $(\widehat{\psi},\alpha)\in S_1'$ iff $\alpha$ is an assignment for $\widehat{\psi}$ and one of the conditions a,b,c above holds. This concludes the inductive step.

It is left to be checked that so defined $S_1'$ satisfies the chosen finite $A$ of the theory. Conditions i,ii,iii and v follow easily by construction. To verify iv we first observe that for every $x$, every subformula $\psi$ of $F(x)$ is a sentence, $\emptyset$ is the unique assignment for $\psi$ and $\widehat{\emptyset}_{\psi}$ is constant. By induction on the complexity of $\phi\in B$ we check that whenever $\phi$ is an $F$-intermediate formula, then 
\begin{equation}\label{equat_ind_cond_reg}\tag{$*$}
   (\phi,\emptyset)\in S_1' \iff \widehat{\phi} \textnormal{ is $X$-satisfied by } \widehat{\emptyset}_{\phi}.
\end{equation}

This is clearly the case for formulae of minimal complexity. We consider the induction step for $\phi = \psi_0 \vee \psi_1$. If it is not  the case that both $\psi_0,\psi_1$ are in $B$, then the claim follows by definition. So assume $\psi_0$ and $\psi_1$ are both in $B$. Hence
\[(\phi,\emptyset)\in S_1'\iff (\psi_0,\emptyset)\in S_1' \textnormal{ or } (\psi_1,\emptyset)\in S_1'.\]
By the inductive assumption, the last condition is equivalent to: 
\begin{equation}\label{equat_ind_2}\tag{$**$}
    \widehat{\emptyset}_{\psi_0} X-\textnormal{satisfies } \widehat{\psi_0} \textnormal{ or } \widehat{\emptyset}_{\psi_1} X-\textnormal{satisfies }\widehat{\psi_1}.
\end{equation} 
Let $\kappa^0$ be the occurrence of $\psi_0$ in $\phi$ as the left disjunct, and $\kappa^1$ be the occurrence of $\psi_1$ in $\phi$ as the right disjunct. Then $(\kappa^0)_{\widehat{\phi}}$ differs from $\widehat{\psi}$ only up to the bound variables renaming and a permutation of free variables. Let $\sigma$ be the permutation of free variables such that $\sigma[(\kappa^{0})_{\widehat{\phi}}]$ is (up to bounded variables renaming) the same as $\widehat{\psi_0}$. By unraveling the definitions it follows that $\widehat{\emptyset}_{\phi}\restr{(\kappa^{0})_{\widehat{\phi}}} = \widehat{\emptyset}_{\psi_0}\circ \sigma$. The same holds for the pair $(\kappa^{1})_{\widehat{\phi}}$ and $\widehat{\psi_1}$. So we conclude that  \eqref{equat_ind_2} is equivalent to 
\[\widehat{\emptyset}_{\phi}\restr{(\kappa^{0})_{\widehat{\phi}}} X-\textnormal{satisfies }(\kappa^{0})_{\widehat{\phi}} \textnormal{ or } \widehat{\emptyset}_{\phi}\restr{(\kappa^{1})_{\widehat{\phi}}} X-\textnormal{satisfies }(\kappa^{1})_{\widehat{\phi}}.\]
The above however is clearly equivalent to the right-hand side of \eqref{equat_ind_cond_reg}.
\end{proof}

\bibliographystyle{plain}
\bibliography{refs}

\begin{thebibliography}{10}

\bibitem{cieslinski-lelyk-wcislo-dc}
Cezary Cie{\'s}li{\'n}ski, Mateusz {\L}e{\l}yk, and Bartosz Wcis{\l}o.
\newblock The two halves of disjunctive correctness.
\newblock {\em Journal of Mathematical Logic},
  https://doi.org/10.1142/S021906132250026X, 2022.

\bibitem{enayat-pakhomov}
Ali Enayat and Fedor Pakhomov.
\newblock Truth, disjunction, and induction.
\newblock {\em Archive for Mathematical Logic}, 58(5):753--766, 2019.

\bibitem{enayat-visser}
Ali Enayat and Albert Visser.
\newblock New constructions of satisfaction classes.
\newblock In {\em Unifying the philosophy of truth}, volume~36 of {\em Log.
  Epistemol. Unity Sci.}, pages 321--335. Springer, Dordrecht, 2015.

\bibitem{elw}
Ali Enayat, Mateusz Łełyk, and Bartosz Wcisło.
\newblock Truth and feasible reducibility.
\newblock {\em The Journal of Symbolic Logic}, 85(1):367–421, 2020.

\bibitem{kossak-omega}
Roman Kossak.
\newblock Models with the {$\omega$}-property.
\newblock {\em Journal of Symbolic Logic}, 54(1):177–189, 1989.

\bibitem{ks}
Roman Kossak and James~H. Schmerl.
\newblock {\em The structure of models of {P}eano arithmetic}, volume~50 of
  {\em Oxford Logic Guides}.
\newblock Oxford University Press, 2006.

\bibitem{kkl}
Henryk Kotlarski, Stanislav Krajewski, and Alistair~H Lachlan.
\newblock Construction of satisfaction classes for nonstandard models.
\newblock {\em Canadian mathematical bulletin}, 24(3):283--293, 1981.

\bibitem{wcislyk_mpt}
Mateusz \L{}e\l{}yk and Bartosz Wcis\l{}o.
\newblock Models of positive truth.
\newblock {\em Review of Symbolic Logic}, 12(1):144--172, 2019.

\bibitem{smith-nonstandard-def}
Stuart~T. Smith.
\newblock Nonstandard definability.
\newblock {\em Annals of Pure and Applied Logic}, 42(1):21--43, 1989.

\bibitem{wcislo_satisfaction_definability}
Bartosz Wcis\l{}o.
\newblock Full satisfaction classes, definability, and automorphisms.
\newblock {\em Notre Dame Journal of Formal Logic}, 63(2), 2022.

\bibitem{wcislyk_notes_on_bounded}
Bartosz Wcis\l{}o and Mateusz \L{}e\l{}yk.
\newblock Notes on bounded induction for the compositional truth predicate.
\newblock {\em Review of Symbolic Logic}, 10(3):455--480, 2017.

\end{thebibliography}

\end{document}